\documentclass[12pt]{amsart}

\usepackage{amsmath}
\usepackage{amssymb}
\usepackage{mathdots}
\usepackage{amsbsy}
\usepackage{amscd}
\usepackage{amsthm}
\usepackage{mathrsfs}
\usepackage{verbatim}
\usepackage[colorlinks]{hyperref}
\usepackage{fullpage}
\usepackage[english]{babel}
\usepackage{url}
\usepackage[colorinlistoftodos]{todonotes}
\usepackage{graphics}

\theoremstyle{definition}
\newtheorem{theorem}{Theorem}[section]
\newtheorem{lemma}[theorem]{Lemma}

\newtheorem{remark}[theorem]{Remark}
\newtheorem{definition}[theorem]{Definition}
\newtheorem{corollary}[theorem]{Corollary}
\newtheorem{proposition}[theorem]{Proposition}

\newcommand{\ZZ}{\mathbb{Z}}
\newcommand{\QQ}{\mathbb{Q}}
\newcommand{\RR}{\mathbb{R}}
\newcommand{\CC}{\mathbb{C}}
\newcommand{\HH}{\mathbb{H}}
\newcommand{\OO}{\mathbb{O}}
\newcommand{\PP}{\mathbb{P}}

\newcommand{\wil}[1]{\widetilde{#1}}
\newcommand{\conj}[1]{\overline{#1}}

\newcommand{\mt}[1]{\text{#1}}
\newcommand{\mf}[1]{\mathfrak{#1}}

\newcommand{\oi}{\textbf{i}}
\newcommand{\oj}{\textbf{j}}
\newcommand{\ok}{\textbf{k}}
\newcommand{\ol}{\textbf{l}}
\newcommand{\oli}{\textbf{li}}
\newcommand{\olj}{\textbf{lj}}
\newcommand{\olk}{\textbf{lk}}

\newcommand{\re}{\text{Re}}
\newcommand{\im}{\text{Im}}
\newcommand{\alt}{\text{Alt}}
\newcommand{\xmin}{X}
\newcommand{\sing}{\Sigma}
\newcommand{\singlong}{\text{Sing($\xmin$)}}

\newcommand{\qdcrpr}{x \times u \times v \times w}

\begin{document}

\title{On the Complex Cayley Grassmannian}
\author{\"Ust\"un Y\i ld\i r\i m}
\address{Department of Mathematics, Michigan State University, MI, 48824}
\email{ustun@mailbox.org}
\date{\today}

\begin{abstract}
  We define a torus action on the (complex) Cayley Grassmannian $X$.
  Using this action, we prove that $X$ is a singular variety.
  We also show that the singular locus is smooth and has the same cohomology ring as that of $\CC\PP^5$.
  Furthermore, we identify the singular locus with a quotient of $G_2^\CC$ by a parabolic subgroup.
\end{abstract}

\date{}
\maketitle

\section{Introduction}
On octonions $\OO\cong\CC^8$, one can define a 4-fold cross product which we denote by \linebreak $(\cdot\times\cdot\times\cdot\times\cdot)\in\Lambda^4\OO^*$ \cite{HL82}.
A 4-plane $\xi= \langle u,v,w,z\rangle$ is called a Cayley plane if it satisfies $\im(u\times v\times w\times z)=0$.
We call the space of all Cayley planes the Cayley Grassmannian denoted by $X$.
It is well known that using octonions over the real numbers with a positive definite metric, the Cayley Grassmannian is isomorphic to the ordinary Grassmannian of 3-planes in 7-space, Gr$(3,7)$.
However, the complex analogue of this space is not well studied.
Using the Pl\"ucker relations and the above description of Cayley planes one can show that $X$ is a 12-dimensional subvariety of the complex 4-planes in 8-space, Gr$(4,\OO)$, see Section~\ref{sec:charts}.

The group Spin$(7,\CC)$ admits an action on $X$ and by restriction to a maximal torus, we get a $(\CC^*)^3$ action on $X$.
It turns out this action has finitely many fixed points, see Theorem~\ref{thm:fixpnts}.
Furthermore, some of the fixed points are singular showing that $X$ is a singular variety, see Theorem~\ref{thm:fxdpnts}.
The singular locus $\Sigma=Sing(X)$ is, however, a smooth projective variety.
Indeed, it can be identified as $G_2^\CC/P_2$, see Theorem~\ref{thm:parabolicquotient}.
Here, $P_2$ is a parabolic subgroup of $G_2^\CC$ corresponding to a long root of $G_2^\CC$.

The structure of this paper mainly follows that of \cite{AC15}.  Whereas we investigate the Cayley Grassmannian, \cite{AC15} investigates associative Grassmannian. The main difficulty in our case is that the Cayley Grassmannian is not a smooth variety.

In Section \ref{sec:oct}, we give an overview of composition algebras. Then, we write down a concrete description of octonions $\OO$ and define various cross product operations using octonionic multiplication.

In Section \ref{sec:cayley}, we define associative and Cayley calibration forms. Then, we express the Cayley calibration and the imaginary part of a four-fold cross product in coordinates.

In Section \ref{sec:charts}, we describe charts of Gr(4,8) when viewed as a subvariety of $\PP(\Lambda^4\CC^8)$. We also give seven linear equations on $\PP(\Lambda^4\CC^8)$ that the Cayley Grassmannian satisfy. Then, we localize these equations to a chart, using Pl\"ucker relations.

In Section \ref{sec:spinsevenandthreesls}, we give a definition of Spin(7,$\CC$), describe three of its subgroups each of which is isomorphic to SL(2,$\CC$) and also describe a maximal torus.

In Section \ref{sec:torusfxd}, we diagonalize the maximal torus and list the eigenvalues and corresponding eigenvectors of the induced action on $\Lambda^4\CC$. Then, we find which eigenvectors lie in $X$. We also prove that the fixed point set of the torus action is finite.

In section \ref{sec:mincompt}, we show that all but six of the torus fixed points are regular, proving $X$ is singular.

In section \ref{sec:singloc}, we prove that the singular locus is smooth and has the same cohomology ring as that of $\CC\PP^5$. Furthermore, we express the singular locus as a quotient of $G_2^{\CC}$ by a parabolic subgroup.

In section \ref{sec:toruslocalsmooth}, we describe the torus action near regular fixed points using a one parameter subgroup of the maximal torus.

\subsection*{\bf Acknowledgement.} I would like to thank Mahir Bilen Can for many inspirational ideas and discussions that not only improved this work but also me as a mathematician. 

\section{Octonions and multiple cross products}\label{sec:oct}
In this section, we recall the definition of octonions and multiple cross products defined using octonionic multiplication. Although octonions are not necessary to define these cross products, they give an alternative way to verify their properties using octonions. Further details can be found in \cite{SV13,HL82,AC15,Bae02,SW10}.

\begin{definition}[\cite{SV13}]
A composition algebra $C$ over a field $k$ is an algebra over $k$ with identity element and a nondegenerate quadratic form $N$ such that
\begin{equation*}
  N(uv)=N(u)N(v)
\end{equation*}
for $u,v\in C$.
The quadratic form $N$ is often referred to as the norm on $C$, and the associated bilinear form $B(\cdot,\cdot)$ is called the inner product.
\end{definition}
In this paper, we take the base field $k$ to be $\CC$.
However, it is not necessary for some parts of the discussion.

A four-dimensional composition algebra is called a quaternion algebra.
Let $\HH$ be a copy of $\CC^4$ generated by $1,\oi, \oj$ and $\ok$ with the relations $\oi^2=\oj^2=\ok^2=\oi\oj\ok= -1$, and $N$ be the quadratic form given by
\begin{equation*}
N(a1+b\oi+c\oj+d\ok)=a^2+b^2+c^2+d^2.
\end{equation*}
Then, $\HH$ is a quaternion algebra over $\CC$.
Following ~\cite{SV13}, we use Cayley-Dickinson doubling to construct an eight-dimensional composition algebra $\OO$ (called octonion algebra) with the new generator $\ol$ whose norm is $1$, i.e., $N(\ol)=1$.
We use the following figure to describe the multiplication table.
\begin{center}
\includegraphics{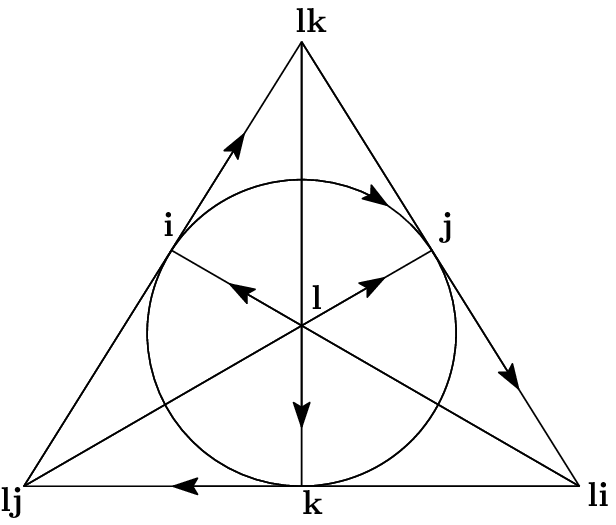}
\end{center}
For each (oriented) line from $x$ to $y$ to $z$, we have the relations
\begin{equation*}
  xy=z, \qquad
  yz=x, \qquad
  zx=y, \qquad \text{and} \qquad
  x^2=y^2=z^2=-1.
\end{equation*}
To ease notation later on, we set $e_0 = 1, e_1 = \oi, e_2 = \oj,$ $e_3 = \ok, e_4=\ol, e_5=\oli, e_6=\olj$ and $e_7=\olk$.
We also set $e^{pqrs}=e^p\wedge e^q\wedge e^r\wedge e^s$ where $\left\{ e^p \right\}$ is the dual basis of $\{e_p\}$.

The octonions are non-associative but they are alternative, i.e., the subalgebra generated by any two elements is associative.
We denote the projection map from $\OO$ to the span of $1$ by $\re$, and 
projection to the orthogonal complement $1^{\perp}$ by $\im$.
This allows us to define an involution $$u\mapsto\conj u = \re(u)-\im(u).$$
The bilinear form $B$ associated to $N$ can be expressed as $B(u,v)=\re(\conj u v)$.
So $N(u)=B(u,u)=\re(\conj u u)$.

A key fact one can verify on the basis elements is
\begin{equation}
  \conj {uv} = \conj v\, \conj u.
  \label{eq:conjofproduct}
\end{equation}
Note that (\ref{eq:conjofproduct}) implies
\begin{equation*}
  \conj{\conj u u} = \conj u u
\end{equation*}
that is $\conj u u\in\re(\OO)$.
Thus,
\begin{equation*}
  N(uv)=\re(\conj{uv} uv ) =\re(\conj v(\conj u u) v) = \re(\conj u u) \re(\conj vv) = N(u)N(v)
\end{equation*}
proving directly that the above multiplication table defines an eight-dimensional composition algebra.
Over $\CC$, there is only one such algebra, namely the octonions.

\begin{lemma}
  For $u,v,v'\in\OO$,
  \begin{equation}
    N(u)B(v,v') = B(uv,uv') = B(vu,v'u).
  \end{equation}
  In particular, for unit $u$, (left or right) multiplication by $u$ is an orthogonal transformation of $\OO$.
  \label{lem:orth}
\end{lemma}
\begin{proof}
  Since $B(v,v')=\frac{1}{2}\left( N(v+v')-N(v)-N(v') \right)$, we have
  \begin{eqnarray*}
    N(u)B(v,v') &=& \frac{1}{2}\left( N(u)N(v+v')-N(u)N(v)-N(u)N(v') \right) \\
    &=& \frac{1}{2}\left( N(uv+uv')-N(uv)-N(uv') \right) \\
    &=& B(uv,uv').
  \end{eqnarray*}
  The second equality can be proved similarly.
\end{proof}

Next, we recall the definition of an $r$-fold cross product.
\begin{definition}
  \label{defn:crsprd}
  Let $(V,B)$ be a vector space with a (non-degenerate) symmetric bilinear form.
  A multilinear map $L:V^r\to V$ is called an $r$-fold cross product if 
  \begin{equation}
  N(L(v_1,\dots,v_n)) = N(v_1\wedge\dots\wedge v_n)
    \label{eq:nfldpro1}
  \end{equation}
with the induced norm on $\Lambda^nV$ and
  \begin{equation}
    B(L(v_1,\dots,v_n),v_i) = 0 \qquad \text{for all } i.
    \label{eq:nfldpro2}
  \end{equation}
\end{definition}
\begin{remark}
  If $L$ is an alternating multilinear map, then it is enough to check (\ref{eq:nfldpro1}) on orthogonal vectors in which case (\ref{eq:nfldpro1}) becomes
  \begin{equation}
    N(L(v_1,\dots,v_n))=N(v_1)\dots N(v_n).
    \label{eq:nfldpro3}
  \end{equation}
  \label{rm:altpro}
\end{remark}
\begin{remark}
  The usual cross product operation on $\RR^3$ is naturally a (two-fold) cross product according to this definition.
\end{remark}

In \cite{BG67}, Brown and Gray proved that an $r$-fold cross product exists on an $n$-dimensional vector space only in the following cases:
\begin{enumerate}
  \item $n$ is even, $r=1$
  \item $n$ is arbitrary, $r=n-1$
  \item $n=3$ or 7, $r=2$
  \item $n=4$ or 8, $r=3$.
\end{enumerate}
Below, we give concrete description of the ``exceptional'' two-fold and three-fold cross products using octonions.
Then, we introduce a ``four-fold cross product'' operation on $\OO$.
Although it is not a cross product according to Definition \ref{defn:crsprd}, it is conventionally called so \cite{HL82,SW10}.
A two-fold cross product (or simply a cross product) can be defined as the restriction of octonionic multiplication to the imaginary part, $\im(\OO)$:
\begin{definition}
  For $u,v\in\im(\OO)$, let 
  \begin{equation}
  u\times v = \im(uv).
  \end{equation}
\end{definition}
\begin{remark}
  By restricting to $u,v\in\im(\HH)$, one gets the two-fold cross product on the three-dimensional space $\im(\HH)$.
\end{remark}

\begin{proposition}
  The map $(u,v)\mapsto u\times v = \im(uv)$ is a two-fold cross product on $\im(\OO)$.
\end{proposition}
\begin{proof}
  Since $\conj u u\in\re(\OO)$, $u\times u = \im(uu)=-\im(\conj uu)=0$.
  So, $u\times v$ is an alternating map.
  By Remark \ref{rm:altpro}, we may assume $u,v\in\im(\OO)$ are orthogonal, that is, $B(u,v)=0$.
  Then, by Lemma \ref{lem:orth} we get
  \begin{eqnarray*}
    0 &=& N(u)B(u,v) \\
    &=& -B(\conj uu, uv) \\
    &=& -B(N(u), uv).
  \end{eqnarray*}
  Thus, $uv\in\im(\OO)$.
  This gives us 
  \begin{eqnarray*}
    N(u\times v) &=& N(\im(uv)) \\
    &=& N(uv) \\
    &=& N(u)N(v).
  \end{eqnarray*}
  To prove (\ref{eq:nfldpro2}), we once again use Lemma \ref{lem:orth}.
  \begin{eqnarray*}
    B(u\times v,u) &=& B(\im(uv),u) \\
    &=& B(uv-\re(uv),u) \\
    &=& B(uv,u)-B(\re(uv),u) \\
    &=& N(u)B(v,1) \\
    &=& 0
  \end{eqnarray*}
  since $u$ and $v$ are orthogonal to $1$.
\end{proof}
\begin{remark}
  There are a number of ways to express the (two-fold) cross product.
  It is easy to show that $\im(uv)=\im(\conj vu)=\alt(\conj vu)$ where $\alt(L)$ is alternation of the multilinear map $L:V^r\to V$. More precisely,
  \begin{equation*}
  \alt(L)(v_1,\dots,v_r) = \frac{1}{r!}\sum_{\sigma\in S_r} \text{sign}(\sigma)L(v_{\sigma(1)},\dots,v_{\sigma(r)})
  \end{equation*}
  where $S_r$ is the symmetric group on the set $\left\{ 1,\dots,r \right\}$.
\end{remark}
This cross product operation is of vital importance in $G_2$ geometry.
Next, following \cite{SW10} we define three-fold and four-fold cross product operations as follows:
\begin{definition}
  For $u,v,w\in\OO$, let
  \begin{equation}
    u\times v \times w = \frac{1}{2}\left( (u\conj v)w - (w\conj v)u \right).
    \label{eq:tcrpr}
  \end{equation}
\end{definition}
\begin{definition}
  For $u,v,w,x\in\OO$, let
\begin{equation}
  x \times u\times v \times w = -\frac{1}{4}\left[(x \times u\times v)\conj w
    - (w \times x\times u)\conj v
    + (v \times w\times x)\conj u
    - (u \times v\times w)\conj x \right].
  \label{eq:qcrpr}
\end{equation}
\end{definition}
\begin{remark}
  The four-fold cross product operation (\ref{eq:qcrpr}) does not satisfy (\ref{eq:nfldpro2}).
  However, it is alternating and satisfies (\ref{eq:nfldpro1}).
  Hence, for orthogonal vectors $x,u,v,$ and $w$ we have
\begin{equation}
  N(x \times u \times v \times w) = N(x)N(u)N(v)N(w).
  \label{eq:qdcrprnorm}
\end{equation}
\end{remark}

\section{Cayley four-planes}\label{sec:cayley}
A $k$-form $\omega\in\Lambda^kV^*$ is called a calibration form if for every orthonormal set of vectors $\left\{ v_1,\dots,v_k \right\}$ we have $|\omega(v_1,\dots,v_n)|\le 1$.
Given a $k$-plane $\xi$ generated by an orthonormal basis $\left\{ v_1,\dots,v_k \right\}$, $\xi$ is called calibrated if $\omega(v_1,\dots,v_k)=\pm 1$.
Using the two-fold (resp. three-fold) cross product, we define a calibration three-form (resp. four-form) $\phi$ (resp. $\Phi$) called associative (resp. Cayley) calibration on $\im(\OO)$ (resp. $\OO$) as follows:
\begin{definition}
  For $u,v,w\in\im(\OO)$, let
  \begin{equation}
    \phi(u,v,w)=B(u, v\times w)
    \label{eq:phi}
  \end{equation}
  and for $x,u,v,w\in\OO$, let
  \begin{equation}
    \Phi(x,u,v,w)=B(x,u\times v\times w).
    \label{eq:Phi}
  \end{equation}
\end{definition}

For a proof of the following proposition see \cite{SW10}.
\begin{proposition} The equations (\ref{eq:phi}) and (\ref{eq:Phi}) define calibration forms and they satisfy
  \begin{equation}
   \phi(u,v,w)=\re(u\times v\times w) 
    \label{eq:phicrpr}
  \end{equation}
  and
  \begin{equation}
    \Phi(x,u,v,w)=\re(x\times u\times v\times w).
    \label{eq:Phicrpr}
  \end{equation}
\end{proposition}

By (\ref{eq:qdcrprnorm}) and (\ref{eq:Phicrpr}), it is clear that $\Phi(x,u,v,w)=\pm1$ if and only if $\Xi(x,u,v,w):=\im(\qdcrpr)=0$ for orthonormal $x,u,v,$ and $w$.
\begin{definition}
A four-plane $\xi$ generated by $\left\{x,u,v,w\right\}$ is called a Cayley plane if \linebreak $\Xi(x,u,v,w)=0$.
The set of all Cayley planes is called the Cayley Grassmannian and denoted $\xmin$.
We denote the set of all planes $\xi$ generated by orthonormal $\left\{ x,u,v,w \right\}$ satisfying $\Phi(x,u,v,w)=\pm 1$ by $\xmin^0$.
\end{definition}
\begin{remark}
  If a four-plane $\xi$ is generated by orthonormal vectors $\left\{ x,u,v,w \right\}$, then \linebreak $\Xi(x,u,v,w)=0$ is equivalent to $\Phi(x,u,v,w)=\pm 1$.
  Therefore, over $\RR$ with a positive definite metric, there is no distinction between the two conditions and $\xmin^0=\xmin$.
  In fact, they show $\xmin^0=Gr(3,7)$ which is a 12-dimensional variety in \cite{HL82}.
  However, over $\CC$, not every four-plane is generated by an orthonormal basis (with respect to $B$).
  Therefore, we have $\xmin^0\subset\xmin$ in general.
  Furthermore, $\xmin^0$ can be thought of as the set of generic points in $\xmin$.
\end{remark}

It is helpful to express $\phi$, $\Phi$ and $\Xi$ in coordinates.
The associative calibration form is given by
\begin{equation*}
\phi = e^{123}-e^{145}-e^{167}-e^{246}+e^{257}-e^{347}-e^{356}
\end{equation*}
the Cayley calibration form is given by
\begin{eqnarray*}
\Phi&=& e^{0123}-e^{0145}-e^{0167}-e^{0246}+e^{0257}-e^{0347}-e^{0356}
\\ & & -e^{1247}-e^{1256}+e^{1346}-e^{1357}-e^{2345}-e^{2367}+e^{4567}
\end{eqnarray*}
and the imaginary part of the four-fold cross product is given by
\begin{eqnarray}
  \Xi =
  \nonumber
& \left( -e^{0247} -e^{0256} +e^{0346} -e^{0357} +e^{1246} -e^{1257} +e^{1347} +e^{1356}\right)e_{1}\\
  \nonumber
& + \left( +e^{0147} +e^{0156} -e^{0345} -e^{0367} -e^{1245} -e^{1267} +e^{2347} +e^{2356}\right)e_{2}\\
  \nonumber
& + \left( -e^{0146} +e^{0157} +e^{0245} +e^{0267} -e^{1345} -e^{1367} -e^{2346} +e^{2357}\right)e_{3}\\
  \label{eq:xi}
& + \left( -e^{0127} +e^{0136} -e^{0235} +e^{0567} +e^{1234} -e^{1467} +e^{2457} -e^{3456}\right)e_{4}\\
  \nonumber
& + \left( -e^{0126} -e^{0137} +e^{0234} -e^{0467} +e^{1235} -e^{1567} +e^{2456} +e^{3457}\right)e_{5}\\
  \nonumber
& + \left( +e^{0125} -e^{0134} -e^{0237} +e^{0457} +e^{1236} -e^{1456} -e^{2567} +e^{3467}\right)e_{6}\\
  \nonumber
& + \left( +e^{0124} +e^{0135} +e^{0236} -e^{0456} +e^{1237} -e^{1457} -e^{2467} -e^{3567}\right)e_{7}.
  \nonumber
\end{eqnarray}
Using these expressions, one can immediately see that
\begin{equation}
\Phi = e^0\wedge \phi + *\phi
\label{eq:AssocAndCayleyCal}
\end{equation}
where $*$ is the Hodge star operator.

\section{Charts of Gr(4,\texorpdfstring{$\OO$}{\bf O})}\label{sec:charts}
Set $p_{ijkl}=e^{ijkl}$ so they are coordinate functions on $\Lambda^4\OO$.
We also view $p_{ijkl}$ as homogeneous coordinates on $\PP (\Lambda^4\OO)$.
The Grassmannian of four-planes in eight dimensions Gr(4,$\OO$) is the variety given by the Pl\"ucker relations (see \cite{KL72}):
\begin{equation}
  p_{i_1i_2i_3j_1}p_{j_2j_3j_4j_5}= p_{i_1i_2i_3j_2}p_{j_1j_3j_4j_5}-p_{i_1i_2i_3j_3}p_{j_1j_2j_4j_5}+p_{i_1i_2i_3j_4}p_{j_1j_2j_3j_5}-p_{i_1i_2i_3j_5}p_{j_1j_2j_3j_4}
  \label{eq:plucker}
\end{equation}
We consider the intersection of $\mt{Gr}(4,\OO)$ and the zero locus of $\Xi$.
By (\ref{eq:xi}), the zero locus of $\Xi$ is given by these seven linear equations:
\begin{eqnarray}
  \label{eq:plucker1}
  f_1:= -p_{0247} -p_{0256} +p_{0346} -p_{0357} +p_{1246} -p_{1257} +p_{1347} +p_{1356} &=& 0 \\
  f_2:= +p_{0147} +p_{0156} -p_{0345} -p_{0367} -p_{1245} -p_{1267} +p_{2347} +p_{2356} &=& 0 \\
  f_3:= -p_{0146} +p_{0157} +p_{0245} +p_{0267} -p_{1345} -p_{1367} -p_{2346} +p_{2357} &=& 0 \\
  f_4:= -p_{0127} +p_{0136} -p_{0235} +p_{0567} +p_{1234} -p_{1467} +p_{2457} -p_{3456} &=& 0 \\
  f_5:= -p_{0126} -p_{0137} +p_{0234} -p_{0467} +p_{1235} -p_{1567} +p_{2456} +p_{3457} &=& 0 \\
  f_6:= +p_{0125} -p_{0134} -p_{0237} +p_{0457} +p_{1236} -p_{1456} -p_{2567} +p_{3467} &=& 0 \\
  f_7:= +p_{0124} +p_{0135} +p_{0236} -p_{0456} +p_{1237} -p_{1457} -p_{2467} -p_{3567} &=& 0 .
  \label{eq:plucker7}
\end{eqnarray}
This intersection is the Cayley Grassmannian.

Once we choose a chart $U_{stun}=\left\{ x\in\PP(\Lambda^4\OO) \;|\; p_{stun}(x)\neq 0 \right\}$, we use the following notation for local coordinates (suppressing the indices $s,t,u,n$).
\begin{equation*}
  q_{ijkl} = \frac{p_{ijkl}}{p_{stun}}.
\end{equation*}
For example, over $U_{0123}$, using Pl\"ucker relations (\ref{eq:plucker}), we have
\begin{equation}
  \frac{p_{4567}}{p_{0123}} = \frac{p_{0456}}{p_{0123}}\frac{p_{1237}}{p_{0123}} - \frac{p_{1456}}{p_{0123}}\frac{p_{0237}}{p_{0123}} + \frac{p_{2456}}{p_{0123}}\frac{p_{0137}}{p_{0123}}-\frac{p_{3456}}{p_{0123}}\frac{p_{0127}}{p_{0123}}
  \label{eq:pluckerexample}
\end{equation}
or, more concisely,
\begin{equation}
  q_{4567} = q_{0456}q_{1237}-q_{1456}q_{0237}+q_{2456}q_{0137}-q_{3456}q_{0127}.
  \label{eq:pluckerexampleq}
\end{equation}

Fix a chart $U_{stun}$.
Then, one can show that any coordinate function can be expressed only in terms of the variables $q_{ijkl}$ with $|\{i,j,k,l\}\cap\{s,t,u,n\}| = 3$ by using (\ref{eq:plucker}) (repeatedly if necessary).
There are exactly 16 such variables corresponding to the fact that $\dim(\text{Gr}(4,\OO))=16$ and they give us the charts of Gr(4,$\OO$), see \cite{KL72,AC15}.

For example, on $U_{0123}$, the local (Grassmannian) variables are
$$\left\{ q_{0124},q_{0125},q_{0126},q_{0127},q_{0134},q_{0135},q_{0136},q_{0137},q_{0234},q_{0235},q_{0236},q_{0237},q_{1234},q_{1235},q_{1236},q_{1237} \right\}.$$
We localize the defining equations (\ref{eq:plucker1})-(\ref{eq:plucker7}) of Cayley planes to these coordinates.
\begin{eqnarray*}
f_1 &=& - q_{0124} q_{0237} + q_{0124} q_{1236} - q_{0125} q_{0236} - q_{0125} q_{1237} + q_{0126} q_{0235} - q_{0126} q_{1234} + q_{0127} q_{0234} \\
    & & + q_{0127} q_{1235} + q_{0134} q_{0236} + q_{0134} q_{1237} - q_{0135} q_{0237} + q_{0135} q_{1236} - q_{0136} q_{0234} - q_{0136} q_{1235} \\
    & & + q_{0137} q_{0235} - q_{0137} q_{1234} \\
f_2 &=& q_{0124} q_{0137} - q_{0124} q_{1235} + q_{0125} q_{0136} + q_{0125} q_{1234} - q_{0126} q_{0135} - q_{0126} q_{1237} - q_{0127} q_{0134} \\
    & & + q_{0127} q_{1236} - q_{0134} q_{0235} + q_{0135} q_{0234} - q_{0136} q_{0237} + q_{0137} q_{0236} + q_{0234} q_{1237} + q_{0235} q_{1236} \\
    & & - q_{0236} q_{1235} - q_{0237} q_{1234}\\
f_3 &=& - q_{0124} q_{0136} + q_{0124} q_{0235} + q_{0125} q_{0137} - q_{0125} q_{0234} + q_{0126} q_{0134} + q_{0126} q_{0237} - q_{0127} q_{0135} \\
    & & - q_{0127} q_{0236} - q_{0134} q_{1235} + q_{0135} q_{1234} - q_{0136} q_{1237} + q_{0137} q_{1236} - q_{0234} q_{1236} + q_{0235} q_{1237} \\
    & & + q_{0236} q_{1234} - q_{0237} q_{1235}
\end{eqnarray*}
\begin{eqnarray*}
f_4 &=& - q_{0124} q_{0136} q_{1237} + q_{0124} q_{0137} q_{1236} + q_{0124} q_{0235} q_{1237} - q_{0124} q_{0237} q_{1235} \\
    & & + q_{0125} q_{0136} q_{0237} - q_{0125} q_{0137} q_{0236} - q_{0125} q_{0234} q_{1237} + q_{0125} q_{0237} q_{1234} \\
    & & + q_{0126} q_{0134} q_{1237} - q_{0126} q_{0135} q_{0237} + q_{0126} q_{0137} q_{0235} - q_{0126} q_{0137} q_{1234} \\
    & & - q_{0127} - q_{0127} q_{0134} q_{1236} + q_{0127} q_{0135} q_{0236} - q_{0127} q_{0136} q_{0235} + q_{0127} q_{0136} q_{1234}\\
    & & + q_{0127} q_{0234} q_{1235} - q_{0127} q_{0235} q_{1234} - q_{0134} q_{0235} q_{1236} + q_{0134} q_{0236} q_{1235} \\
    & & + q_{0135} q_{0234} q_{1236} - q_{0135} q_{0236} q_{1234} + q_{0136} - q_{0136} q_{0234} q_{1235} + q_{0136} q_{0235} q_{1234} - q_{0235} + q_{1234} \\
f_5 &=& - q_{0124} q_{0136} q_{0237} + q_{0124} q_{0137} q_{0236} + q_{0124} q_{0235} q_{1236} - q_{0124} q_{0236} q_{1235} \\
    & &  - q_{0125} q_{0136} q_{1237} + q_{0125} q_{0137} q_{1236} - q_{0125} q_{0234} q_{1236} + q_{0125} q_{0236} q_{1234} \\
    & & - q_{0126} + q_{0126} q_{0134} q_{0237} + q_{0126} q_{0135} q_{1237} - q_{0126} q_{0137} q_{0234} - q_{0126} q_{0137} q_{1235} \\
    & & + q_{0126} q_{0234} q_{1235} - q_{0126} q_{0235} q_{1234} - q_{0127} q_{0134} q_{0236} - q_{0127} q_{0135} q_{1236} \\
    & & + q_{0127} q_{0136} q_{0234} + q_{0127} q_{0136} q_{1235} + q_{0134} q_{0235} q_{1237} - q_{0134} q_{0237} q_{1235} \\
    & & - q_{0135} q_{0234} q_{1237} + q_{0135} q_{0237} q_{1234} - q_{0137} + q_{0137} q_{0234} q_{1235} - q_{0137} q_{0235} q_{1234} + q_{0234} + q_{1235} \\
f_6 &=& q_{0124} q_{0135} q_{0237} - q_{0124} q_{0135} q_{1236} + q_{0124} q_{0136} q_{1235} - q_{0124} q_{0137} q_{0235} \\
    & & + q_{0125} - q_{0125} q_{0134} q_{0237} + q_{0125} q_{0134} q_{1236} - q_{0125} q_{0136} q_{1234} + q_{0125} q_{0137} q_{0234} \\
    & & - q_{0125} q_{0236} q_{1237} + q_{0125} q_{0237} q_{1236} - q_{0126} q_{0134} q_{1235} + q_{0126} q_{0135} q_{1234} \\
    & & + q_{0126} q_{0235} q_{1237} - q_{0126} q_{0237} q_{1235} + q_{0127} q_{0134} q_{0235} - q_{0127} q_{0135} q_{0234} \\
    & & - q_{0127} q_{0235} q_{1236} + q_{0127} q_{0236} q_{1235} - q_{0134} + q_{0134} q_{0236} q_{1237} - q_{0134} q_{0237} q_{1236} \\
    & & - q_{0136} q_{0234} q_{1237} + q_{0136} q_{0237} q_{1234} + q_{0137} q_{0234} q_{1236} - q_{0137} q_{0236} q_{1234} - q_{0237} + q_{1236} \\
f_7 &=& q_{0124} - q_{0124} q_{0135} q_{0236} - q_{0124} q_{0135} q_{1237} + q_{0124} q_{0136} q_{0235} + q_{0124} q_{0137} q_{1235} \\
    & & - q_{0124} q_{0236} q_{1237} + q_{0124} q_{0237} q_{1236} + q_{0125} q_{0134} q_{0236} + q_{0125} q_{0134} q_{1237} \\
    & & - q_{0125} q_{0136} q_{0234} - q_{0125} q_{0137} q_{1234} - q_{0126} q_{0134} q_{0235} + q_{0126} q_{0135} q_{0234} \\ 
    & & + q_{0126} q_{0234} q_{1237} - q_{0126} q_{0237} q_{1234} - q_{0127} q_{0134} q_{1235} + q_{0127} q_{0135} q_{1234} \\
    & & - q_{0127} q_{0234} q_{1236} + q_{0127} q_{0236} q_{1234} + q_{0135} - q_{0135} q_{0236} q_{1237} + q_{0135} q_{0237} q_{1236} \\
    & & + q_{0136} q_{0235} q_{1237} - q_{0136} q_{0237} q_{1235} - q_{0137} q_{0235} q_{1236} + q_{0137} q_{0236} q_{1235} + q_{0236} + q_{1237}
\end{eqnarray*}
\section{Spin(7,\texorpdfstring{$\CC$}{\bf C}) and Three SL(2,\texorpdfstring{$\CC$}{\bf C}) actions}\label{sec:spinsevenandthreesls}
In this section, we first give an unusual definition of the group Spin(7,$\CC$) and $G_2^{\CC}$ following Bryant \cite{Bry87}.
It is well known that Spin(7,$\CC$) is of rank three, that is, its maximal tori are three-dimensional.
Then, we describe a maximal torus of Spin(7,$\CC$) and three subgroups of Spin(7,$\CC$) each of which is isomorphic to SL(2,$\CC$).

\begin{definition}
  Spin$(7,\CC)$ is the identity component of $\{M\in \mt{SO}(8,\CC)\ |\ M^* \Phi = \Phi\}$ and $G_2^\CC = \left\{ M\in \mt{SO}(7,\CC) \ |\  M^*\phi = \phi \right\}$.
\end{definition}
Given $A\in G_2^\CC$, we can extend it linearly so that it fixes $1\in \OO$.
Then, it is easy to see that $G_2^\CC$ is a subgroup of $\mt{Spin}(7,\CC)$ using (\ref{eq:AssocAndCayleyCal}).

By definition, an element of Spin(7,$\CC$) acts on $\OO$ preserving orthonormality and the values of $\Phi$.
Thus, it takes a Cayley plane to a Cayley plane.
In other words, it defines an action on the Cayley Grassmannian.

 Consider the following matrix
\begin{eqnarray*}
  L_\lambda = \left(
  \begin{matrix}
    P_\lambda & -iM_\lambda \\
    iM_\lambda & P_\lambda
  \end{matrix}
  \right)
\end{eqnarray*}
where $P_\lambda=\frac{\lambda+\lambda^{-1}}{2}, M_\lambda=\frac{\lambda-\lambda^{-1}}{2},$ and $\lambda\in\CC^*$.
Note that its determinant is $1$.
In fact, it has eigenvalues $\lambda$ and $\lambda^{-1}$ with eigenvectors
\[
  \begin{pmatrix}
    1 \\
    i
  \end{pmatrix}
  \quad \mt{ and }
  \quad 
  \begin{pmatrix}
    1 \\
    -i
  \end{pmatrix}
\]
respectively.

Considering $L_\lambda$ as a block matrix, we define the following $8\times8$ matrices $A_\lambda = L_\lambda \oplus L_{\lambda}\oplus L_\lambda \oplus L_{\lambda},$ $B_\mu = L_\mu\oplus L_{\mu^{-1}} \oplus I_4$ and $C_\gamma = I_4\oplus L_\gamma\oplus L_{\gamma^{-1}} $ and view them as transformations of $\OO$ with respect to the standard basis $\left\{ e_i \right\}$ where $I_n$ is the $n\times n$ identity matrix.

\begin{lemma}
 The image of $h:(\CC^*)^3\to $SL(8,$\CC$) defined by
\begin{equation*}
  h(\lambda,\mu,\gamma) = A_\lambda B_\mu C_\gamma
\end{equation*}
is a maximal torus $T$ of Spin(7,$\CC$).
  \label{lm:maxtor}
\end{lemma}
\begin{proof}
  It is easy to prove that $L_\lambda L_\mu = L_{\lambda\mu}$ and if $L_\lambda = I_2$ then $\lambda=1$.
  It follows that $A_\lambda,B_\mu$ and $C_\gamma$ commute with each other.
  Hence, $h$ is a well defined homomorphism.
  Furthermore, the kernel of $h$ is given by $\left\{ \pm(1,1,1) \right\}$ and thus, the image $T \cong (\CC^*)^3/\ZZ_2 $ is isomorphic to $(\CC^*)^3$.
  Since the rank of Spin(7,$\CC$) is $3$, we only need to show that $T\subset$ Spin(7,$\CC$).
  
  A simple computation shows that $(L_\lambda)^{-1}=L_{\lambda^{-1}}=L_\lambda^T$.
  In other words, $L_\lambda\in$ SO$(2,\CC)$ which implies $T\subset$ SO$(8,\CC)$.
  Finally, we need to show that for $M\in T$, $M^*\Phi=\Phi$.
  We verify this by a direct computation with the help of a software.
\end{proof}

From our discussion above for $L_\lambda$, it is easy to find eigenvalues and eigenvectors for $h(\lambda,\mu,\gamma)$.
They are given in the following table.
\begin{center}
  \begin{equation}
    \label{tbl:eiva}
    \begin{tabular}[]{c|c}
      eigenvalue                 &  eigenvector       \\  \hline
      $\lambda\mu$               &  $1+i\oi$          \\  \hline
      $\lambda^{-1}\mu^{-1}$     &  $1-i\oi$          \\  \hline
      $\lambda\mu^{-1}$          &  $\oj+i\ok$        \\  \hline
      $\lambda^{-1}\mu$          &  $\oj-i\ok$        \\  \hline
      $\lambda\gamma$            &  $\ol+i\oli$     \\  \hline
      $\lambda^{-1}\gamma^{-1}$  &  $\ol-i\oli$     \\  \hline
      $\lambda\gamma^{-1}$       &  $\olj+i\olk$  \\  \hline
      $\lambda^{-1}\gamma$       &  $\olj-i\olk$
    \end{tabular}
  \end{equation}
\end{center}

We identify $\mt{SL}(2,\CC)$ as the subgroup of the multiplicative group of $\HH$ with $N=1$.
More explicitly, $u\in\HH$ is identified with the matrix $A_u$ given by
\[
  u=a1+b\oi+c\oj+d\ok \mapsto
  A_u = 
  \begin{pmatrix}
    a - id & -b + ic\\
    b + ic & a + id
  \end{pmatrix}.
\]
Note that $A_u:\HH\to \CC^{2\times2}$ is a linear isomorphism and it satisfies
\[
  A_uA_v = A_{uv}.
\]
Moreover, $N(u) = \det(A_u)$.
Hence, $(\CC^{2\times 2}, \det)$ is a quaternion algebra isomorphic to $(\HH,N)$ via $(u\mapsto A_u)$.
Thus, $\mt{SL}(2,\CC)$ can be identified with the unit sphere of $\HH$, i.e., $\left\{ v\in \HH\;|\; N(v)=1 \right\}$.

\begin{proposition}
There are three SL(2,$\CC$) actions on $\OO$ which preserve $B$ and $\Phi$.
To describe these actions we express $\OO$ as a direct sum $\OO = \HH \oplus \ol\HH $.
Let $v=(x,y)\in\OO$.
\begin{enumerate}
  \item $g\cdot v = (xg^{-1},y)$
    \label{actone}
  \item $g\cdot v = (x,yg^{-1})$
  \item $g\cdot v = (gx,gy)$
\end{enumerate}
\end{proposition}
\begin{proof}
  Since $g\in$ SL(2,$\CC$) is identified with an element of $\HH$ with norm 1, multiplication by $g$ is an orthogonal transformation by Lemma \ref{lem:orth}.
  Thus, $B$ is preserved in all three actions.

  To show that $\Phi$ is also preserved, we instead look at the corresponding action of the Lie algebra $\mathfrak{sl}(2,\CC)\cong \im(\HH)=\langle\oi,\oj,\ok\rangle$.
  It is enough to show that $\oi\cdot\Phi=\oj\cdot\Phi=\ok\cdot\Phi=0$.
  We verify this by a direct computation for all three actions.
\end{proof}

\begin{remark}
Note that all three actions are faithful and thus, provide three different embeddings of SL$(2,\CC)$ into Spin(7,$\CC$).
Furthermore, we can define an action of the group $(\mt{SL}(2,\CC))^3$ on $\OO$ by
\begin{equation}
(a,b,c)\cdot(x,y) = (cxa^{-1}, cyb^{-1})
\label{eq:allactionscombined}
\end{equation}
for $a,b,c\in \mt{SL}(2,\CC)$ and $(x,y)\in\OO$.
The kernel of this action is $\left\{ \pm(1,1,1) \right\}$.
\label{rm:slembeddings}
\end{remark}

The matrices corresponding to the unipotent element $\begin{pmatrix} 1 & u \\ 0 & 1 \end{pmatrix}\in SL(2,\CC)$ (under three actions) are $\begin{bmatrix} A & 0 \\ 0 & I_4 \end{bmatrix}$, $\begin{bmatrix} I_4 & 0 \\ 0 & A \end{bmatrix}$ and $\begin{bmatrix} B & 0 \\ 0 & B \end{bmatrix}$ where
\begin{equation*}
  A = 
  \begin{pmatrix}
    1 & 0 & \frac{-u}{2} & \frac{iu}{2} \\
    0 & 1 & \frac{iu}{2} & \frac{u}{2}  \\
    \frac{u}{2} & \frac{-iu}{2} & 1 & 0  \\
    \frac{-iu}{2} & \frac{-u}{2} & 0 & 1 
  \end{pmatrix}
\end{equation*}
and
\begin{equation*}
  B = 
  \begin{pmatrix}
    1 & 0 & 0 & 0 \\
    0 & 1 & iu & u \\
    0 & -iu & 1+\frac{u^2}{2} & \frac{-iu^2}{2} \\
    0 & -u & \frac{-iu^2}{2} & 1-\frac{u^2}{2}
  \end{pmatrix}.
\end{equation*}

Clearly, the SL(2,$\CC$) actions preserve the splitting $\OO = \HH \oplus \ol\HH$.
Hence, the points $[e_{0123}],[e_{4567}]\in \xmin$ are fixed by all three of these actions. 
Restricting to the unipotent group inside SL(2,$\CC$), we conclude that fixed point sets (of the unipotent subgroup) are positive dimensional by a result of Horrocks \cite{Hor69}.

\section{Torus fixed points}\label{sec:torusfxd}
We set $\wil e_0=  1+i\oi, \wil e_1=  1-i\oi, \wil e_2=  \oj+i\ok, \wil e_3=  \oj-i\ok, \wil e_4=  \ol+i\oli, \wil e_5=  \ol-i\oli, \wil e_6=  \olj+i\olk, $ and $\wil e_7 = \olj -i \olk$.
Recall from (\ref{tbl:eiva}) that these vectors are the eigenvectors of the matrix $h(\lambda,\mu,\gamma)$. 
We also set $\wil e_{pqrs}=\wil e_p\wedge \wil e_q\wedge\wil e_r\wedge\wil e_s$.
If we denote by $\wil p_{pqrs}$ the transformed Pl\"ucker coordinates, the equations (\ref{eq:plucker1})-(\ref{eq:plucker7}) can be rewritten as
\begin{eqnarray}
  \label{eq:transformedplucker1}
\wil f_1 &:=& \wil p_{0257} - \wil p_{1346}                                                                                                 =0\\
\wil f_2 &:=& \wil p_{0146} - \wil p_{0157} - \wil p_{0245} - \wil p_{0267} + \wil p_{1345} + \wil p_{1367} + \wil p_{2346} - \wil p_{2357} =0\\
\wil f_3 &:=& \wil p_{0146} + \wil p_{0157} - \wil p_{0245} - \wil p_{0267} - \wil p_{1345} - \wil p_{1367} + \wil p_{2346} + \wil p_{2357} =0\\
\wil f_4 &:=& \wil p_{0127} - \wil p_{0136} + \wil p_{0235} - \wil p_{0567} - \wil p_{1234} + \wil p_{1467} - \wil p_{2457} + \wil p_{3456} =0\\
\wil f_5 &:=& \wil p_{0127} + \wil p_{0136} - \wil p_{0235} + \wil p_{0567} - \wil p_{1234} + \wil p_{1467} - \wil p_{2457} - \wil p_{3456} =0\\
\wil f_6 &:=& \wil p_{0125} - \wil p_{0134} - \wil p_{0237} + \wil p_{0457} + \wil p_{1236} - \wil p_{1456} - \wil p_{2567} + \wil p_{3467} =0\\
\wil f_7 &:=& \wil p_{0125} + \wil p_{0134} + \wil p_{0237} - \wil p_{0457} + \wil p_{1236} - \wil p_{1456} - \wil p_{2567} - \wil p_{3467} =0
  \label{eq:transformedplucker7}
\end{eqnarray}

Now that we diagonalized the action, it is easy to describe the eigenvalues and the corresponding eigenvectors of the action of $h(\lambda,\mu,\gamma)$ on $\Lambda^4 \OO$.
\begin{center}
  \begin{tabular}[]{c|c}
    eigenvalue & eigenvector \\ \hline
 $1$ & $\wil e_{0123}, \wil e_{0145}, \wil e_{0167}, \wil e_{0257}, \wil e_{1346}, \wil e_{2345}, \wil e_{2367}, \wil e_{4567}$ \\ \hline
 $\gamma^{-2}$ & $\wil e_{0156}, \wil e_{2356}$ \\ \hline
 $\gamma^{-2} \lambda^{-2}$ & $\wil e_{1356}$ \\ \hline
 $\gamma^{-2} \lambda^{2}$ & $\wil e_{0256}$ \\ \hline
 $\gamma^{-2} \mu^{-2}$ & $\wil e_{1256}$ \\ \hline
 $\gamma^{-2} \mu^{2}$ & $\wil e_{0356}$ \\ \hline
 $\gamma^{-1} \lambda^{-2} \mu^{-1}$ & $\wil e_{1235}, \wil e_{1567}$ \\ \hline
 $\gamma^{-1} \lambda^{-2} \mu$ & $\wil e_{0135}, \wil e_{3567}$ \\ \hline
 $\gamma^{-1} \lambda^{2} \mu^{-1}$ & $\wil e_{0126}, \wil e_{2456}$ \\ \hline
 $\gamma^{-1} \lambda^{2} \mu$ & $\wil e_{0236}, \wil e_{0456}$ \\ \hline
 $\gamma^{-1} \mu^{-1}$ & $\wil e_{0125}, \wil e_{1236}, \wil e_{1456}, \wil e_{2567}$ \\ \hline
 $\gamma^{-1} \mu$ & $\wil e_{0136}, \wil e_{0235}, \wil e_{0567}, \wil e_{3456}$ \\ \hline
 $\gamma \lambda^{-2} \mu^{-1}$ & $\wil e_{1237}, \wil e_{1457}$ \\ \hline
 $\gamma \lambda^{-2} \mu$ & $\wil e_{0137}, \wil e_{3457}$ \\ \hline
 $\gamma \lambda^{2} \mu^{-1}$ & $\wil e_{0124}, \wil e_{2467}$ \\ \hline
 $\gamma \lambda^{2} \mu$ & $\wil e_{0234}, \wil e_{0467}$ \\ \hline
 $\gamma \mu^{-1}$ & $\wil e_{0127}, \wil e_{1234}, \wil e_{1467}, \wil e_{2457}$ \\ \hline
 $\gamma \mu$ & $\wil e_{0134}, \wil e_{0237}, \wil e_{0457}, \wil e_{3467}$ \\ \hline
 $\gamma^{2}$ & $\wil e_{0147}, \wil e_{2347}$ \\ \hline
 $\gamma^{2} \lambda^{-2}$ & $\wil e_{1347}$ \\ \hline
 $\gamma^{2} \lambda^{2}$ & $\wil e_{0247}$ \\ \hline
 $\gamma^{2} \mu^{-2}$ & $\wil e_{1247}$ \\ \hline
 $\gamma^{2} \mu^{2}$ & $\wil e_{0347}$ \\ \hline
   \end{tabular}
   \end{center}
   \begin{center}
     \begin{tabular}[]{c|c}
      eigenvalue & eigenvector \\ \hline
 $\lambda^{-4}$ & $\wil e_{1357}$ \\ \hline
 $\lambda^{-2}$ & $\wil e_{0157}, \wil e_{1345}, \wil e_{1367}, \wil e_{2357}$ \\ \hline
 $\lambda^{-2} \mu^{-2}$ & $\wil e_{1257}$ \\ \hline
 $\lambda^{-2} \mu^{2}$ & $\wil e_{0357}$ \\ \hline
 $\lambda^{2}$ & $\wil e_{0146}, \wil e_{0245}, \wil e_{0267}, \wil e_{2346}$ \\ \hline
 $\lambda^{2} \mu^{-2}$ & $\wil e_{1246}$ \\ \hline
 $\lambda^{2} \mu^{2}$ & $\wil e_{0346}$ \\ \hline
 $\lambda^{4}$ & $\wil e_{0246}$ \\ \hline
 $\mu^{-2}$ & $\wil e_{1245}, \wil e_{1267}$ \\ \hline
 $\mu^{2}$ & $\wil e_{0345}, \wil e_{0367}$ 
  \end{tabular}
\end{center}

\begin{theorem}
  The following eigenvectors of $h(\lambda,\mu,\gamma)$ lie in $\xmin$. 
  \label{thm:eive}
  \begin{center}
    \begin{tabular}[]{c|c}
      eigenvalue & eigenvector \\ \hline
 $1$ & $\wil e_{0123}, \wil e_{0145}, \wil e_{0167}, \wil e_{2345}, \wil e_{2367}, \wil e_{4567}$ \\ \hline
 $\gamma^{-2}$ & $\wil e_{0156}, \wil e_{2356}$ \\ \hline
 $\gamma^{-2} \lambda^{-2}$ & $\wil e_{1356}$ \\ \hline
 $\gamma^{-2} \lambda^{2}$ & $\wil e_{0256}$ \\ \hline
 $\gamma^{-2} \mu^{-2}$ & $\wil e_{1256}$ \\ \hline
 $\gamma^{-2} \mu^{2}$ & $\wil e_{0356}$ \\ \hline
 $\gamma^{-1} \lambda^{-2} \mu^{-1}$ & $\wil e_{1235}, \wil e_{1567}$ \\ \hline
 $\gamma^{-1} \lambda^{-2} \mu$ & $\wil e_{0135}, \wil e_{3567}$ \\ \hline
 $\gamma^{-1} \lambda^{2} \mu^{-1}$ & $\wil e_{0126}, \wil e_{2456}$ \\ \hline
 $\gamma^{-1} \lambda^{2} \mu$ & $\wil e_{0236}, \wil e_{0456}$ \\ \hline
 $\gamma \lambda^{-2} \mu^{-1}$ & $\wil e_{1237}, \wil e_{1457}$ \\ \hline
 $\gamma \lambda^{-2} \mu$ & $\wil e_{0137}, \wil e_{3457}$ \\ \hline
 $\gamma \lambda^{2} \mu^{-1}$ & $\wil e_{0124}, \wil e_{2467}$ \\ \hline
 $\gamma \lambda^{2} \mu$ & $\wil e_{0234}, \wil e_{0467}$ \\ \hline
 $\gamma^{2}$ & $\wil e_{0147}, \wil e_{2347}$ \\ \hline
 $\gamma^{2} \lambda^{-2}$ & $\wil e_{1347}$ \\ \hline
 $\gamma^{2} \lambda^{2}$ & $\wil e_{0247}$ \\ \hline
 $\gamma^{2} \mu^{-2}$ & $\wil e_{1247}$ \\ \hline
 $\gamma^{2} \mu^{2}$ & $\wil e_{0347}$ \\ \hline
 $\lambda^{-4}$ & $\wil e_{1357}$ \\ \hline
 $\lambda^{-2} \mu^{-2}$ & $\wil e_{1257}$ \\ \hline
 $\lambda^{-2} \mu^{2}$ & $\wil e_{0357}$ \\ \hline
 $\lambda^{2} \mu^{-2}$ & $\wil e_{1246}$ \\ \hline
 $\lambda^{2} \mu^{2}$ & $\wil e_{0346}$ \\ \hline
 $\lambda^{4}$ & $\wil e_{0246}$ \\ \hline
 $\mu^{-2}$ & $\wil e_{1245}, \wil e_{1267}$ \\ \hline
 $\mu^{2}$ & $\wil e_{0345}, \wil e_{0367}$ 
    \end{tabular}
  \end{center}
\end{theorem}
\begin{proof}
  Once we evaluate all the eigenvectors on the transformed defining equations (\ref{eq:transformedplucker1})-(\ref{eq:transformedplucker7}), we see that exactly the above list of vectors satisfy them.
\end{proof}
\begin{theorem}
  The fixed point set $\xmin^{T}$ of the maximal torus action is only the above set of points.
  \label{thm:fixpnts}
\end{theorem}
\begin{proof}
  We only need to verify that in eigenspaces of dimension greater than one, there are no other fixed points.
  We prove this for the eigenspace associated to the eigenvalue 1.
  Let 
  $$\wil x = \wil c_{0123}\wil e_{0123}+\wil c_{0145}\wil e_{0145}+\wil c_{0167}\wil e_{0167}+\wil c_{2345}\wil e_{2345}+\wil c_{2367}\wil e_{2367}+\wil c_{4567}\wil e_{4567}\in \xmin$$
  be a torus fixed point that is different from
 $\wil e_{0123}, \wil e_{0145}, \wil e_{0167}, \wil e_{2345}, \wil e_{2367},$ and $\wil e_{4567}$.
 Then, at least two of the coordinates $\wil c_I,$ and $\wil c_J$ are nonzero for index sets $I\neq J$.
 Let $I=\{i_1i_2i_3j_1\}$, and $J=\{j_2j_3j_4j_5\}$.
 If $\wil x\in\xmin \subset$ Gr(4,$\OO$) than it has to satisfy the Pl\"ucker relation (\ref{eq:plucker})
 \[
  \wil c_{i_1i_2i_3j_1}\wil c_{j_2j_3j_4j_5}= \wil c_{i_1i_2i_3j_2}\wil c_{j_1j_3j_4j_5}-\wil c_{i_1i_2i_3j_3}\wil c_{j_1j_2j_4j_5}+\wil c_{i_1i_2i_3j_4}\wil c_{j_1j_2j_3j_5}-\wil c_{i_1i_2i_3j_5}\wil c_{j_1j_2j_3j_4}.
 \]
 Now, the left hand side of this relation is nonzero by choice.
 However, the right hand side has to be zero as the index sets of eigenvectors associated to eigenvalue 1 differ by at least two elements.
 This means that $\wil x$ does not belong to $\mt{Gr}(4,\OO)$ and hence, $\wil x\notin \xmin$.

 This argument also works for other eigenspaces of dimension greater than one.
\end{proof}

\section{The Cayley Grassmannian}\label{sec:mincompt}
We cite the following well known theorem, here.
\begin{theorem}[Jacobian Criterion for Smoothness]\label{thm:jacobian}
  Let $I=(f_1,\dots,f_m)$ be an ideal from $\CC[x_1,\dots,x_n]$ and let $x\in V(I)$ be a point from the vanishing locus of $I$ in $\CC^n$.
  Suppose $d=\dim V(I)$.
  If the rank of the Jacobian matrix $(\partial f_i/\partial x_j)_{i=1,\dots,m,\ j=1,\dots,n}$ at $x$ is equal to $n-d$, then $x$ is a smooth point of $V(I)$.
\end{theorem}
Torus action takes singular locus to singular locus.
Since the singular locus is also a projective variety, it must contain a torus fixed point by Borel fixed point theorem.
Therefore, we can check whether $\xmin$ is smooth or not, by using the Jacobian criterion on the torus fixed points.
\begin{theorem}\label{thm:fxdpnts}
  Among the fixed points listed in Theorem \ref{thm:eive}, all but the following six of them are smooth points.
  \[
    \wil e_{0246}, \wil e_{0347}, \wil e_{0356}, \wil e_{1247}, \wil e_{1256}, \wil e_{1357}
  \]
\end{theorem}
\begin{proof}
  We need to analyze neighborhoods of fixed points by using affine charts.
  We start with the fixed point $m=\wil e_{0123}$, which lies on the open chart $\wil U_{0123}$ as its origin.
  Recall that $\xmin$ is cut-out on $\wil U_{0123}$ by the vanishing of the seven linear equations (\ref{eq:transformedplucker1})-(\ref{eq:transformedplucker7}).
  At first, it may seem that it is necessary to express these equations in local variables $\wil q_{0124},\wil q_{0125},\wil q_{0126},\wil q_{0127},\wil q_{0134},\wil q_{0135},\wil q_{0136},\wil q_{0137},$ $\wil q_{0234}, \wil q_{0235},\wil q_{0236},\wil q_{0237},\wil q_{1234},\wil q_{1235},\wil q_{1236},$ and $\wil q_{1237}$.
  However, the Pl\"ucker relations (in local coordinates) replace linear terms with higher order terms (see, for example, (\ref{eq:pluckerexampleq})) and we compute Jacobian at the origin.
  So, there will be no contribution to Jacobian matrix from other variables.
  The Jacobian matrix at $m$ is given by
  \begin{eqnarray*}
    \wil J_{0123} &=&  
    \begin{pmatrix}
      0 & 0 & 0 & 0 & 0 & 0 & 0 & 0 & 0 & 0 & 0 & 0 & 0 & 0 & 0 & 0 \\ 
      0 & 0 & 0 & 0 & 0 & 0 & 0 & 0 & 0 & 0 & 0 & 0 & 0 & 0 & 0 & 0 \\ 
      0 & 0 & 0 & 0 & 0 & 0 & 0 & 0 & 0 & 0 & 0 & 0 & 0 & 0 & 0 & 0 \\ 
      0 & 0 & 0 & 1 & 0 & 0 & -1 & 0 & 0 & 1 & 0 & 0 & -1 & 0 & 0 & 0 \\ 
      0 & 0 & 0 & 1 & 0 & 0 & 1 & 0 & 0 & -1 & 0 & 0 & -1 & 0 & 0 & 0 \\ 
      0 & 1 & 0 & 0 & -1 & 0 & 0 & 0 & 0 & 0 & 0 & -1 & 0 & 0 & 1 & 0 \\ 
      0 & 1 & 0 & 0 & 1 & 0 & 0 & 0 & 0 & 0 & 0 & 1 & 0 & 0 & 1 & 0
    \end{pmatrix}
  \end{eqnarray*}
  which is of rank four. So, it is at most 12-dimensional. However, it contains the 12-dimensional $\xmin^0$. So, it must be 12-dimensional, that is, codimension four in Gr(4,$\OO$). Hence, by Theorem~\ref{thm:jacobian}, $\wil e_{0123}$ is a smooth point of $\xmin$.

  We repeat this computation for the other points and see that their Jacobian matrices are all rank four, except for the six points we have listed above.
  Thus, they are all smooth points of $\xmin$.
\end{proof}


\section{Singular Locus}\label{sec:singloc}
Next, we turn our attention to the singular locus $\sing:=\singlong$ of $\xmin$.
The torus action on $\xmin$ restricts to $\sing$.
By Theorem~\ref{thm:fxdpnts}, we know that there are six points in $\sing^T$.
We quote the following lemma from \cite{BCM02}.
\begin{lemma}
  \label{lem:cntFxdPnts}
  If $Y\subset \mathbb P(V)$ is a projective $T$-variety, then $Y^T$ contains at least dim $Y+1$ points.
\end{lemma}
Therefore, we have the following corollary.
\begin{corollary}
The singular locus $\sing$ is at most five-dimensional.
\end{corollary}
As we did with $\xmin$, we can check whether $\sing$ is singular or not, by using Jacobian criterion on these six torus fixed points $\sing^T$.
\begin{theorem}
  \label{thm:singLocSmooth}
  The singular locus $\sing$ is smooth and five-dimensional.
\end{theorem}
\begin{proof}
  We start with the point $\wil e_{0246}$.
  This point lies at the origin of the chart $\wil U_{0246}=\left\{x\in\PP(\Lambda^4\OO)\;|\; \wil p_{0246}(x)\neq 0\right\}$.
  On this chart the local variables are $\wil q_{0124},$ $\wil q_{0126},$ $\wil q_{0146},$ $\wil q_{0234},$ $\wil q_{0236},$ $\wil q_{0245},$ $\wil q_{0247},$ $\wil q_{0256},$ $\wil q_{0267},$ $\wil q_{0346},$ $\wil q_{0456},$ $\wil q_{0467},$ $\wil q_{1246},$ $\wil q_{2346},$ $\wil q_{2456},$ and $\wil q_{2467}$ where $\wil q_{ijkl}=\wil p_{ijkl}/\wil p_{0246}$.
  Since the codimension of $\xmin$ in $\text{Gr}(4,\OO)$ is four, $\sing$ is locally the vanishing locus of the equations (\ref{eq:transformedplucker1})-(\ref{eq:transformedplucker7}) localized to $\wil U_{0246}$ and all $4\times 4$ minors of the Jacobian of these localized equations.
  These equations by themselves do not generate a radical ideal, so we take the radical ideal generated by those equations with the help of a software called Singular.
  It turns out the ideal is generated by
  \begin{eqnarray*}
    & & \wil q_{1246},\quad 
    \wil q_{0346},\quad 
    \wil q_{0267}- \wil q_{2346},\quad 
    \wil q_{0256},\quad 
    \wil q_{0247},\quad 
    \wil q_{0245}- \wil q_{2346},\quad 
    \wil q_{0236}- \wil q_{0456},\\
    & & \wil q_{0234}- \wil q_{0467},\quad 
    \wil q_{0146}- \wil q_{2346},\quad 
    \wil q_{0126}- \wil q_{2456},\quad 
    \text{and}\quad \wil q_{0124}- \wil q_{2467}.
  \end{eqnarray*}
  So, $\sing$ is just cut out by some hyperplanes in $\wil U_{0246}\cap$Gr(4,$\OO$).
  Therefore, it is clearly smooth at the origin and of dimension five.
  We repeat this computation for the other fixed points and see they are all smooth points of $\sing$. Hence, $\sing$ is smooth and five-dimensional.
\end{proof}

\begin{theorem}
  The  singular locus $\sing$ has the same cohomology ring (over $\QQ$) as $\CC\PP^5$.
  \label{thm:cohomp5}
\end{theorem}
\begin{proof}
  By Theorem \ref{thm:singLocSmooth}, we see that $\sing$ is a smooth projective variety and hence, it is K\"ahler.
  Thus, $2i^{th}$ Betti number is at least one for $i=0,\dots,5$.
  So, the sum of its Betti numbers is at least six.
  On the other hand, there is a torus action on $\sing$ (induced from $T$-action on $\xmin$) with exactly six fixed points.
  Thus, by Bia\l ynicki-Birula decomposition \cite{BB73}, the sum of Betti numbers is exactly six.
\end{proof}

Next, we restrict the action of Spin(7,$\CC$) on $\Sigma$ to the subgroup $G_2^\CC$.
A maximal torus for $G_2^\CC$ is given by $T\cap G_2^\CC$.
Let $\mf g$ denote the Lie algebra of $G_2^\CC$ and $\mf h$ denote the Cartan subalgebra corresponding to our choice of maximal torus.
Choose a set of positive roots $S^{+}$ so that $\wil e_{0246}$ has the highest weight and let $\Delta=\left\{ \alpha_1,\alpha_2 \right\}$ be the set of simple roots corresponding to this choice where $\alpha_2$ is the longer root.
Let $P_i$ be the parabolic subgroup of $G_2^{\CC}$ whose Lie algebra is $\mf g_{-\alpha_i}\bigoplus \mf h \bigoplus\left( \oplus_{\alpha\in S^+} \mf g_{\alpha} \right)$ where $\mf g_{\alpha}=\left\{ X\in \mf g \ |\  [H,X]=\alpha(H)X \ \mt{for all } H\in \mf h\right\}$.
A straight-forward, albeit lengthy, calculation shows that the stabilizer subgroup of $\wil e_{0246}$ is $P_2$.
Hence, we get the following theorem.
\begin{theorem}
$G_2^\CC$ acts on the singular locus $\Sigma$ and the stabilizer group of $\wil e_{0246}$ is $P_2$.
So, by dimensional reasons, $\Sigma = G_2^\CC/P_2$.
\label{thm:parabolicquotient}
\end{theorem}

\section{Torus action near smooth fixed points}\label{sec:toruslocalsmooth}
Note that torus action induces an action on the tangent spaces at the fixed points.
In this section, we choose a regular one parameter subgroup $\tau(\lambda)$ of $T$ and then describe the induced action on these tangent spaces.
The subgroup $\tau(\lambda)$ is regular in the sense that the fixed point set of $T$ and of $\tau(\lambda)$ are the same.

We choose the one parameter subgroup $\tau(\lambda)=h(\lambda,\lambda^{10},\lambda^{100})$ for our computations.
This is regular since it pairs to a non-trivial homomorphism ($\CC^*\to \CC^*$) with any character.
Below, we compute a weight space decomposition for each tangent space.
We list the weight vectors with their corresponding weights and give the number of positive weights.
On a chart $\wil U_{stun}=\left\{ x\in\PP(\Lambda^4\OO)\;|\; \wil p_{stun}(x)\neq 0 \right\}$, we denote the vector field $\frac{\partial}{\partial \wil q_{ijkl}}$ by $\partial_{ijkl}$.

At $\wil e_{0247}$ in $\wil U_{0247}$, the kernel of $\wil J_{0247}$ is generated by the following vectors
$$\begin{tabular}{c|c|c|c|c|c|c}
vector & $\partial_{0124}$ & $\partial_{0147}$ & $\partial_{0234}$ & $\partial_{0246}$ & $\partial_{0347}$ & $\partial_{0467}$\\\hline
weight & $\lambda^{-110}$ & $\lambda^{-2}$ & $\lambda^{-90}$ & $\lambda^{-198}$ & $\lambda^{18}$ & $\lambda^{-90}$
\end{tabular}$$
$$\begin{tabular}{c|c|c|c|c|c|c}
vector & $\partial_{1247}$ & $\partial_{2347}$ & $\partial_{2467}$ & $\partial_{0127} + \partial_{2457}$ & $\partial_{0237} + \partial_{0457}$ & $\partial_{0245} - \partial_{0267}$\\\hline
weight & $\lambda^{-22}$ & $\lambda^{-2}$ & $\lambda^{-110}$ & $\lambda^{-112}$ & $\lambda^{-92}$ & $\lambda^{-200}$
\end{tabular}$$
and the number of positive weights is 1.

At $\wil e_{0147}$ in $\wil U_{0147}$, the kernel of $\wil J_{0147}$ is generated by the following vectors
$$\begin{tabular}{c|c|c|c|c|c|c}
vector & $\partial_{0124}$ & $\partial_{0137}$ & $\partial_{0145}$ & $\partial_{0167}$ & $\partial_{0247}$ & $\partial_{0347}$\\\hline
weight & $\lambda^{-108}$ & $\lambda^{-92}$ & $\lambda^{-200}$ & $\lambda^{-200}$ & $\lambda^{2}$ & $\lambda^{20}$
\end{tabular}$$
$$\begin{tabular}{c|c|c|c|c|c|c}
vector & $\partial_{0467}$ & $\partial_{1247}$ & $\partial_{1347}$ & $\partial_{1457}$ & $\partial_{0127} - \partial_{1467}$ & $\partial_{0134} + \partial_{0457}$\\\hline
weight & $\lambda^{-88}$ & $\lambda^{-20}$ & $\lambda^{-2}$ & $\lambda^{-112}$ & $\lambda^{-110}$ & $\lambda^{-90}$
\end{tabular}$$
and the number of positive weights is 2.

At $\wil e_{2347}$ in $\wil U_{2347}$, the kernel of $\wil J_{2347}$ is generated by the following vectors
$$\begin{tabular}{c|c|c|c|c|c|c}
vector & $\partial_{0234}$ & $\partial_{0247}$ & $\partial_{0347}$ & $\partial_{1237}$ & $\partial_{1247}$ & $\partial_{1347}$\\\hline
weight & $\lambda^{-88}$ & $\lambda^{2}$ & $\lambda^{20}$ & $\lambda^{-112}$ & $\lambda^{-20}$ & $\lambda^{-2}$
\end{tabular}$$
$$\begin{tabular}{c|c|c|c|c|c|c}
vector & $\partial_{2345}$ & $\partial_{2367}$ & $\partial_{2467}$ & $\partial_{3457}$ & $\partial_{0237} + \partial_{3467}$ & $\partial_{1234} - \partial_{2457}$\\\hline
weight & $\lambda^{-200}$ & $\lambda^{-200}$ & $\lambda^{-108}$ & $\lambda^{-92}$ & $\lambda^{-90}$ & $\lambda^{-110}$
\end{tabular}$$
and the number of positive weights is 2.

At $\wil e_{0234}$ in $\wil U_{0234}$, the kernel of $\wil J_{0234}$ is generated by the following vectors
$$\begin{tabular}{c|c|c|c|c|c|c}
vector & $\partial_{0123}$ & $\partial_{0124}$ & $\partial_{0236}$ & $\partial_{0246}$ & $\partial_{0247}$ & $\partial_{0345}$\\\hline
weight & $\lambda^{-112}$ & $\lambda^{-20}$ & $\lambda^{-200}$ & $\lambda^{-108}$ & $\lambda^{90}$ & $\lambda^{-92}$
\end{tabular}$$
$$\begin{tabular}{c|c|c|c|c|c|c}
vector & $\partial_{0346}$ & $\partial_{0347}$ & $\partial_{2345}$ & $\partial_{2347}$ & $\partial_{0134} - \partial_{0237}$ & $\partial_{0245} + \partial_{2346}$\\\hline
weight & $\lambda^{-90}$ & $\lambda^{108}$ & $\lambda^{-112}$ & $\lambda^{88}$ & $\lambda^{-2}$ & $\lambda^{-110}$
\end{tabular}$$
and the number of positive weights is 3.

At $\wil e_{0467}$ in $\wil U_{0467}$, the kernel of $\wil J_{0467}$ is generated by the following vectors
$$\begin{tabular}{c|c|c|c|c|c|c}
vector & $\partial_{0147}$ & $\partial_{0167}$ & $\partial_{0246}$ & $\partial_{0247}$ & $\partial_{0346}$ & $\partial_{0347}$\\\hline
weight & $\lambda^{88}$ & $\lambda^{-112}$ & $\lambda^{-108}$ & $\lambda^{90}$ & $\lambda^{-90}$ & $\lambda^{108}$
\end{tabular}$$
$$\begin{tabular}{c|c|c|c|c|c|c}
vector & $\partial_{0367}$ & $\partial_{0456}$ & $\partial_{2467}$ & $\partial_{4567}$ & $\partial_{0146} + \partial_{0267}$ & $\partial_{0457} - \partial_{3467}$\\\hline
weight & $\lambda^{-92}$ & $\lambda^{-200}$ & $\lambda^{-20}$ & $\lambda^{-112}$ & $\lambda^{-110}$ & $\lambda^{-2}$
\end{tabular}$$
and the number of positive weights is 3.

At $\wil e_{1347}$ in $\wil U_{1347}$, the kernel of $\wil J_{1347}$ is generated by the following vectors
$$\begin{tabular}{c|c|c|c|c|c|c}
vector & $\partial_{0137}$ & $\partial_{0147}$ & $\partial_{0347}$ & $\partial_{1237}$ & $\partial_{1247}$ & $\partial_{1357}$\\\hline
weight & $\lambda^{-90}$ & $\lambda^{2}$ & $\lambda^{22}$ & $\lambda^{-110}$ & $\lambda^{-18}$ & $\lambda^{-202}$
\end{tabular}$$
$$\begin{tabular}{c|c|c|c|c|c|c}
vector & $\partial_{1457}$ & $\partial_{2347}$ & $\partial_{3457}$ & $\partial_{0134} + \partial_{3467}$ & $\partial_{1234} + \partial_{1467}$ & $\partial_{1345} - \partial_{1367}$\\\hline
weight & $\lambda^{-110}$ & $\lambda^{2}$ & $\lambda^{-90}$ & $\lambda^{-88}$ & $\lambda^{-108}$ & $\lambda^{-200}$
\end{tabular}$$
and the number of positive weights is 3.

At $\wil e_{0124}$ in $\wil U_{0124}$, the kernel of $\wil J_{0124}$ is generated by the following vectors
$$\begin{tabular}{c|c|c|c|c|c|c}
vector & $\partial_{0123}$ & $\partial_{0126}$ & $\partial_{0145}$ & $\partial_{0147}$ & $\partial_{0234}$ & $\partial_{0246}$\\\hline
weight & $\lambda^{-92}$ & $\lambda^{-200}$ & $\lambda^{-92}$ & $\lambda^{108}$ & $\lambda^{20}$ & $\lambda^{-88}$
\end{tabular}$$
$$\begin{tabular}{c|c|c|c|c|c|c}
vector & $\partial_{0247}$ & $\partial_{1245}$ & $\partial_{1246}$ & $\partial_{1247}$ & $\partial_{0127} + \partial_{1234}$ & $\partial_{0146} + \partial_{0245}$\\\hline
weight & $\lambda^{110}$ & $\lambda^{-112}$ & $\lambda^{-110}$ & $\lambda^{88}$ & $\lambda^{-2}$ & $\lambda^{-90}$
\end{tabular}$$
and the number of positive weights is 4.

At $\wil e_{0137}$ in $\wil U_{0137}$, the kernel of $\wil J_{0137}$ is generated by the following vectors
$$\begin{tabular}{c|c|c|c|c|c|c}
vector & $\partial_{0123}$ & $\partial_{0135}$ & $\partial_{0147}$ & $\partial_{0167}$ & $\partial_{0347}$ & $\partial_{0357}$\\\hline
weight & $\lambda^{-108}$ & $\lambda^{-200}$ & $\lambda^{92}$ & $\lambda^{-108}$ & $\lambda^{112}$ & $\lambda^{-90}$
\end{tabular}$$
$$\begin{tabular}{c|c|c|c|c|c|c}
vector & $\partial_{0367}$ & $\partial_{1237}$ & $\partial_{1347}$ & $\partial_{1357}$ & $\partial_{0134} - \partial_{0237}$ & $\partial_{0157} + \partial_{1367}$\\\hline
weight & $\lambda^{-88}$ & $\lambda^{-20}$ & $\lambda^{90}$ & $\lambda^{-112}$ & $\lambda^{2}$ & $\lambda^{-110}$
\end{tabular}$$
and the number of positive weights is 4.

At $\wil e_{0346}$ in $\wil U_{0346}$, the kernel of $\wil J_{0346}$ is generated by the following vectors
$$\begin{tabular}{c|c|c|c|c|c|c}
vector & $\partial_{0234}$ & $\partial_{0236}$ & $\partial_{0246}$ & $\partial_{0345}$ & $\partial_{0347}$ & $\partial_{0356}$\\\hline
weight & $\lambda^{90}$ & $\lambda^{-110}$ & $\lambda^{-18}$ & $\lambda^{-2}$ & $\lambda^{198}$ & $\lambda^{-202}$
\end{tabular}$$
$$\begin{tabular}{c|c|c|c|c|c|c}
vector & $\partial_{0367}$ & $\partial_{0456}$ & $\partial_{0467}$ & $\partial_{0134} + \partial_{3467}$ & $\partial_{0136} + \partial_{3456}$ & $\partial_{0146} - \partial_{2346}$\\\hline
weight & $\lambda^{-2}$ & $\lambda^{-110}$ & $\lambda^{90}$ & $\lambda^{88}$ & $\lambda^{-112}$ & $\lambda^{-20}$
\end{tabular}$$
and the number of positive weights is 4.

At $\wil e_{2467}$ in $\wil U_{2467}$, the kernel of $\wil J_{2467}$ is generated by the following vectors
$$\begin{tabular}{c|c|c|c|c|c|c}
vector & $\partial_{0246}$ & $\partial_{0247}$ & $\partial_{0467}$ & $\partial_{1246}$ & $\partial_{1247}$ & $\partial_{1267}$\\\hline
weight & $\lambda^{-88}$ & $\lambda^{110}$ & $\lambda^{20}$ & $\lambda^{-110}$ & $\lambda^{88}$ & $\lambda^{-112}$
\end{tabular}$$
$$\begin{tabular}{c|c|c|c|c|c|c}
vector & $\partial_{2347}$ & $\partial_{2367}$ & $\partial_{2456}$ & $\partial_{4567}$ & $\partial_{0267} + \partial_{2346}$ & $\partial_{1467} + \partial_{2457}$\\\hline
weight & $\lambda^{108}$ & $\lambda^{-92}$ & $\lambda^{-200}$ & $\lambda^{-92}$ & $\lambda^{-90}$ & $\lambda^{-2}$
\end{tabular}$$
and the number of positive weights is 4.

At $\wil e_{3457}$ in $\wil U_{3457}$, the kernel of $\wil J_{3457}$ is generated by the following vectors
$$\begin{tabular}{c|c|c|c|c|c|c}
vector & $\partial_{0345}$ & $\partial_{0347}$ & $\partial_{0357}$ & $\partial_{1347}$ & $\partial_{1357}$ & $\partial_{1457}$\\\hline
weight & $\lambda^{-88}$ & $\lambda^{112}$ & $\lambda^{-90}$ & $\lambda^{90}$ & $\lambda^{-112}$ & $\lambda^{-20}$
\end{tabular}$$
$$\begin{tabular}{c|c|c|c|c|c|c}
vector & $\partial_{2345}$ & $\partial_{2347}$ & $\partial_{3567}$ & $\partial_{4567}$ & $\partial_{0457} - \partial_{3467}$ & $\partial_{1345} + \partial_{2357}$\\\hline
weight & $\lambda^{-108}$ & $\lambda^{92}$ & $\lambda^{-200}$ & $\lambda^{-108}$ & $\lambda^{2}$ & $\lambda^{-110}$
\end{tabular}$$
and the number of positive weights is 4.

At $\wil e_{0345}$ in $\wil U_{0345}$, the kernel of $\wil J_{0345}$ is generated by the following vectors
$$\begin{tabular}{c|c|c|c|c|c|c}
vector & $\partial_{0135}$ & $\partial_{0145}$ & $\partial_{0234}$ & $\partial_{0346}$ & $\partial_{0347}$ & $\partial_{0356}$\\\hline
weight & $\lambda^{-112}$ & $\lambda^{-20}$ & $\lambda^{92}$ & $\lambda^{2}$ & $\lambda^{200}$ & $\lambda^{-200}$
\end{tabular}$$
$$\begin{tabular}{c|c|c|c|c|c|c}
vector & $\partial_{0357}$ & $\partial_{0456}$ & $\partial_{2345}$ & $\partial_{3457}$ & $\partial_{0134} + \partial_{0457}$ & $\partial_{0235} - \partial_{3456}$\\\hline
weight & $\lambda^{-2}$ & $\lambda^{-108}$ & $\lambda^{-20}$ & $\lambda^{88}$ & $\lambda^{90}$ & $\lambda^{-110}$
\end{tabular}$$
and the number of positive weights is 5.

At $\wil e_{0367}$ in $\wil U_{0367}$, the kernel of $\wil J_{0367}$ is generated by the following vectors
$$\begin{tabular}{c|c|c|c|c|c|c}
vector & $\partial_{0137}$ & $\partial_{0167}$ & $\partial_{0236}$ & $\partial_{0346}$ & $\partial_{0347}$ & $\partial_{0356}$\\\hline
weight & $\lambda^{88}$ & $\lambda^{-20}$ & $\lambda^{-108}$ & $\lambda^{2}$ & $\lambda^{200}$ & $\lambda^{-200}$
\end{tabular}$$
$$\begin{tabular}{c|c|c|c|c|c|c}
vector & $\partial_{0357}$ & $\partial_{0467}$ & $\partial_{2367}$ & $\partial_{3567}$ & $\partial_{0136} - \partial_{0567}$ & $\partial_{0237} + \partial_{3467}$\\\hline
weight & $\lambda^{-2}$ & $\lambda^{92}$ & $\lambda^{-20}$ & $\lambda^{-112}$ & $\lambda^{-110}$ & $\lambda^{90}$
\end{tabular}$$
and the number of positive weights is 5.

At $\wil e_{1237}$ in $\wil U_{1237}$, the kernel of $\wil J_{1237}$ is generated by the following vectors
$$\begin{tabular}{c|c|c|c|c|c|c}
vector & $\partial_{0123}$ & $\partial_{0137}$ & $\partial_{1235}$ & $\partial_{1247}$ & $\partial_{1257}$ & $\partial_{1267}$\\\hline
weight & $\lambda^{-88}$ & $\lambda^{20}$ & $\lambda^{-200}$ & $\lambda^{92}$ & $\lambda^{-110}$ & $\lambda^{-108}$
\end{tabular}$$
$$\begin{tabular}{c|c|c|c|c|c|c}
vector & $\partial_{1347}$ & $\partial_{1357}$ & $\partial_{2347}$ & $\partial_{2367}$ & $\partial_{0127} + \partial_{1234}$ & $\partial_{1367} + \partial_{2357}$\\\hline
weight & $\lambda^{110}$ & $\lambda^{-92}$ & $\lambda^{112}$ & $\lambda^{-88}$ & $\lambda^{2}$ & $\lambda^{-90}$
\end{tabular}$$
and the number of positive weights is 5.

At $\wil e_{1457}$ in $\wil U_{1457}$, the kernel of $\wil J_{1457}$ is generated by the following vectors
$$\begin{tabular}{c|c|c|c|c|c|c}
vector & $\partial_{0145}$ & $\partial_{0147}$ & $\partial_{1245}$ & $\partial_{1247}$ & $\partial_{1257}$ & $\partial_{1347}$\\\hline
weight & $\lambda^{-88}$ & $\lambda^{112}$ & $\lambda^{-108}$ & $\lambda^{92}$ & $\lambda^{-110}$ & $\lambda^{110}$
\end{tabular}$$
$$\begin{tabular}{c|c|c|c|c|c|c}
vector & $\partial_{1357}$ & $\partial_{1567}$ & $\partial_{3457}$ & $\partial_{4567}$ & $\partial_{0157} + \partial_{1345}$ & $\partial_{1467} + \partial_{2457}$\\\hline
weight & $\lambda^{-92}$ & $\lambda^{-200}$ & $\lambda^{20}$ & $\lambda^{-88}$ & $\lambda^{-90}$ & $\lambda^{2}$
\end{tabular}$$
and the number of positive weights is 5.

At $\wil e_{0123}$ in $\wil U_{0123}$, the kernel of $\wil J_{0123}$ is generated by the following vectors
$$\begin{tabular}{c|c|c|c|c|c|c}
vector & $\partial_{0124}$ & $\partial_{0126}$ & $\partial_{0135}$ & $\partial_{0137}$ & $\partial_{0234}$ & $\partial_{0236}$\\\hline
weight & $\lambda^{92}$ & $\lambda^{-108}$ & $\lambda^{-92}$ & $\lambda^{108}$ & $\lambda^{112}$ & $\lambda^{-88}$
\end{tabular}$$
$$\begin{tabular}{c|c|c|c|c|c|c}
vector & $\partial_{1235}$ & $\partial_{1237}$ & $\partial_{0125} - \partial_{1236}$ & $\partial_{0127} + \partial_{1234}$ & $\partial_{0134} - \partial_{0237}$ & $\partial_{0136} + \partial_{0235}$\\\hline
weight & $\lambda^{-112}$ & $\lambda^{88}$ & $\lambda^{-110}$ & $\lambda^{90}$ & $\lambda^{110}$ & $\lambda^{-90}$
\end{tabular}$$
and the number of positive weights is 6.

At $\wil e_{0145}$ in $\wil U_{0145}$, the kernel of $\wil J_{0145}$ is generated by the following vectors
$$\begin{tabular}{c|c|c|c|c|c|c}
vector & $\partial_{0124}$ & $\partial_{0135}$ & $\partial_{0147}$ & $\partial_{0156}$ & $\partial_{0345}$ & $\partial_{0456}$\\\hline
weight & $\lambda^{92}$ & $\lambda^{-92}$ & $\lambda^{200}$ & $\lambda^{-200}$ & $\lambda^{20}$ & $\lambda^{-88}$
\end{tabular}$$
$$\begin{tabular}{c|c|c|c|c|c|c}
vector & $\partial_{1245}$ & $\partial_{1457}$ & $\partial_{0125} + \partial_{1456}$ & $\partial_{0134} + \partial_{0457}$ & $\partial_{0146} + \partial_{0245}$ & $\partial_{0157} + \partial_{1345}$\\\hline
weight & $\lambda^{-20}$ & $\lambda^{88}$ & $\lambda^{-110}$ & $\lambda^{110}$ & $\lambda^{2}$ & $\lambda^{-2}$
\end{tabular}$$
and the number of positive weights is 6.

At $\wil e_{0167}$ in $\wil U_{0167}$, the kernel of $\wil J_{0167}$ is generated by the following vectors
$$\begin{tabular}{c|c|c|c|c|c|c}
vector & $\partial_{0126}$ & $\partial_{0137}$ & $\partial_{0147}$ & $\partial_{0156}$ & $\partial_{0367}$ & $\partial_{0467}$\\\hline
weight & $\lambda^{-108}$ & $\lambda^{108}$ & $\lambda^{200}$ & $\lambda^{-200}$ & $\lambda^{20}$ & $\lambda^{112}$
\end{tabular}$$
$$\begin{tabular}{c|c|c|c|c|c|c}
vector & $\partial_{1267}$ & $\partial_{1567}$ & $\partial_{0127} - \partial_{1467}$ & $\partial_{0136} - \partial_{0567}$ & $\partial_{0146} + \partial_{0267}$ & $\partial_{0157} + \partial_{1367}$\\\hline
weight & $\lambda^{-20}$ & $\lambda^{-112}$ & $\lambda^{90}$ & $\lambda^{-90}$ & $\lambda^{2}$ & $\lambda^{-2}$
\end{tabular}$$
and the number of positive weights is 6.

At $\wil e_{0357}$ in $\wil U_{0357}$, the kernel of $\wil J_{0357}$ is generated by the following vectors
$$\begin{tabular}{c|c|c|c|c|c|c}
vector & $\partial_{0135}$ & $\partial_{0137}$ & $\partial_{0345}$ & $\partial_{0347}$ & $\partial_{0356}$ & $\partial_{0367}$\\\hline
weight & $\lambda^{-110}$ & $\lambda^{90}$ & $\lambda^{2}$ & $\lambda^{202}$ & $\lambda^{-198}$ & $\lambda^{2}$
\end{tabular}$$
$$\begin{tabular}{c|c|c|c|c|c|c}
vector & $\partial_{1357}$ & $\partial_{3457}$ & $\partial_{3567}$ & $\partial_{0157} - \partial_{2357}$ & $\partial_{0235} + \partial_{0567}$ & $\partial_{0237} + \partial_{0457}$\\\hline
weight & $\lambda^{-22}$ & $\lambda^{90}$ & $\lambda^{-110}$ & $\lambda^{-20}$ & $\lambda^{-108}$ & $\lambda^{92}$
\end{tabular}$$
and the number of positive weights is 6.

At $\wil e_{1246}$ in $\wil U_{1246}$, the kernel of $\wil J_{1246}$ is generated by the following vectors
$$\begin{tabular}{c|c|c|c|c|c|c}
vector & $\partial_{0124}$ & $\partial_{0126}$ & $\partial_{0246}$ & $\partial_{1245}$ & $\partial_{1247}$ & $\partial_{1256}$\\\hline
weight & $\lambda^{110}$ & $\lambda^{-90}$ & $\lambda^{22}$ & $\lambda^{-2}$ & $\lambda^{198}$ & $\lambda^{-202}$
\end{tabular}$$
$$\begin{tabular}{c|c|c|c|c|c|c}
vector & $\partial_{1267}$ & $\partial_{2456}$ & $\partial_{2467}$ & $\partial_{0146} - \partial_{2346}$ & $\partial_{1234} + \partial_{1467}$ & $\partial_{1236} + \partial_{1456}$\\\hline
weight & $\lambda^{-2}$ & $\lambda^{-90}$ & $\lambda^{110}$ & $\lambda^{20}$ & $\lambda^{108}$ & $\lambda^{-92}$
\end{tabular}$$
and the number of positive weights is 6.

At $\wil e_{2345}$ in $\wil U_{2345}$, the kernel of $\wil J_{2345}$ is generated by the following vectors
$$\begin{tabular}{c|c|c|c|c|c|c}
vector & $\partial_{0234}$ & $\partial_{0345}$ & $\partial_{1235}$ & $\partial_{1245}$ & $\partial_{2347}$ & $\partial_{2356}$\\\hline
weight & $\lambda^{112}$ & $\lambda^{20}$ & $\lambda^{-112}$ & $\lambda^{-20}$ & $\lambda^{200}$ & $\lambda^{-200}$
\end{tabular}$$
$$\begin{tabular}{c|c|c|c|c|c|c}
vector & $\partial_{2456}$ & $\partial_{3457}$ & $\partial_{0235} - \partial_{3456}$ & $\partial_{0245} + \partial_{2346}$ & $\partial_{1234} - \partial_{2457}$ & $\partial_{1345} + \partial_{2357}$\\\hline
weight & $\lambda^{-108}$ & $\lambda^{108}$ & $\lambda^{-90}$ & $\lambda^{2}$ & $\lambda^{90}$ & $\lambda^{-2}$
\end{tabular}$$
and the number of positive weights is 6.

At $\wil e_{2367}$ in $\wil U_{2367}$, the kernel of $\wil J_{2367}$ is generated by the following vectors
$$\begin{tabular}{c|c|c|c|c|c|c}
vector & $\partial_{0236}$ & $\partial_{0367}$ & $\partial_{1237}$ & $\partial_{1267}$ & $\partial_{2347}$ & $\partial_{2356}$\\\hline
weight & $\lambda^{-88}$ & $\lambda^{20}$ & $\lambda^{88}$ & $\lambda^{-20}$ & $\lambda^{200}$ & $\lambda^{-200}$
\end{tabular}$$
$$\begin{tabular}{c|c|c|c|c|c|c}
vector & $\partial_{2467}$ & $\partial_{3567}$ & $\partial_{0237} + \partial_{3467}$ & $\partial_{0267} + \partial_{2346}$ & $\partial_{1236} + \partial_{2567}$ & $\partial_{1367} + \partial_{2357}$\\\hline
weight & $\lambda^{92}$ & $\lambda^{-92}$ & $\lambda^{110}$ & $\lambda^{2}$ & $\lambda^{-110}$ & $\lambda^{-2}$
\end{tabular}$$
and the number of positive weights is 6.

At $\wil e_{4567}$ in $\wil U_{4567}$, the kernel of $\wil J_{4567}$ is generated by the following vectors
$$\begin{tabular}{c|c|c|c|c|c|c}
vector & $\partial_{0456}$ & $\partial_{0467}$ & $\partial_{1457}$ & $\partial_{1567}$ & $\partial_{2456}$ & $\partial_{2467}$\\\hline
weight & $\lambda^{-88}$ & $\lambda^{112}$ & $\lambda^{88}$ & $\lambda^{-112}$ & $\lambda^{-108}$ & $\lambda^{92}$
\end{tabular}$$
$$\begin{tabular}{c|c|c|c|c|c|c}
vector & $\partial_{3457}$ & $\partial_{3567}$ & $\partial_{0457} - \partial_{3467}$ & $\partial_{0567} + \partial_{3456}$ & $\partial_{1456} - \partial_{2567}$ & $\partial_{1467} + \partial_{2457}$\\\hline
weight & $\lambda^{108}$ & $\lambda^{-92}$ & $\lambda^{110}$ & $\lambda^{-90}$ & $\lambda^{-110}$ & $\lambda^{90}$
\end{tabular}$$
and the number of positive weights is 6.

At $\wil e_{0236}$ in $\wil U_{0236}$, the kernel of $\wil J_{0236}$ is generated by the following vectors
$$\begin{tabular}{c|c|c|c|c|c|c}
vector & $\partial_{0123}$ & $\partial_{0126}$ & $\partial_{0234}$ & $\partial_{0246}$ & $\partial_{0256}$ & $\partial_{0346}$\\\hline
weight & $\lambda^{88}$ & $\lambda^{-20}$ & $\lambda^{200}$ & $\lambda^{92}$ & $\lambda^{-110}$ & $\lambda^{110}$
\end{tabular}$$
$$\begin{tabular}{c|c|c|c|c|c|c}
vector & $\partial_{0356}$ & $\partial_{0367}$ & $\partial_{2356}$ & $\partial_{2367}$ & $\partial_{0136} + \partial_{0235}$ & $\partial_{0267} + \partial_{2346}$\\\hline
weight & $\lambda^{-92}$ & $\lambda^{108}$ & $\lambda^{-112}$ & $\lambda^{88}$ & $\lambda^{-2}$ & $\lambda^{90}$
\end{tabular}$$
and the number of positive weights is 7.

At $\wil e_{0456}$ in $\wil U_{0456}$, the kernel of $\wil J_{0456}$ is generated by the following vectors
$$\begin{tabular}{c|c|c|c|c|c|c}
vector & $\partial_{0145}$ & $\partial_{0156}$ & $\partial_{0246}$ & $\partial_{0256}$ & $\partial_{0345}$ & $\partial_{0346}$\\\hline
weight & $\lambda^{88}$ & $\lambda^{-112}$ & $\lambda^{92}$ & $\lambda^{-110}$ & $\lambda^{108}$ & $\lambda^{110}$
\end{tabular}$$
$$\begin{tabular}{c|c|c|c|c|c|c}
vector & $\partial_{0356}$ & $\partial_{0467}$ & $\partial_{2456}$ & $\partial_{4567}$ & $\partial_{0146} + \partial_{0245}$ & $\partial_{0567} + \partial_{3456}$\\\hline
weight & $\lambda^{-92}$ & $\lambda^{200}$ & $\lambda^{-20}$ & $\lambda^{88}$ & $\lambda^{90}$ & $\lambda^{-2}$
\end{tabular}$$
and the number of positive weights is 7.

At $\wil e_{1245}$ in $\wil U_{1245}$, the kernel of $\wil J_{1245}$ is generated by the following vectors
$$\begin{tabular}{c|c|c|c|c|c|c}
vector & $\partial_{0124}$ & $\partial_{0145}$ & $\partial_{1235}$ & $\partial_{1246}$ & $\partial_{1247}$ & $\partial_{1256}$\\\hline
weight & $\lambda^{112}$ & $\lambda^{20}$ & $\lambda^{-92}$ & $\lambda^{2}$ & $\lambda^{200}$ & $\lambda^{-200}$
\end{tabular}$$
$$\begin{tabular}{c|c|c|c|c|c|c}
vector & $\partial_{1257}$ & $\partial_{1457}$ & $\partial_{2345}$ & $\partial_{2456}$ & $\partial_{0125} + \partial_{1456}$ & $\partial_{1234} - \partial_{2457}$\\\hline
weight & $\lambda^{-2}$ & $\lambda^{108}$ & $\lambda^{20}$ & $\lambda^{-88}$ & $\lambda^{-90}$ & $\lambda^{110}$
\end{tabular}$$
and the number of positive weights is 7.

At $\wil e_{1267}$ in $\wil U_{1267}$, the kernel of $\wil J_{1267}$ is generated by the following vectors
$$\begin{tabular}{c|c|c|c|c|c|c}
vector & $\partial_{0126}$ & $\partial_{0167}$ & $\partial_{1237}$ & $\partial_{1246}$ & $\partial_{1247}$ & $\partial_{1256}$\\\hline
weight & $\lambda^{-88}$ & $\lambda^{20}$ & $\lambda^{108}$ & $\lambda^{2}$ & $\lambda^{200}$ & $\lambda^{-200}$
\end{tabular}$$
$$\begin{tabular}{c|c|c|c|c|c|c}
vector & $\partial_{1257}$ & $\partial_{1567}$ & $\partial_{2367}$ & $\partial_{2467}$ & $\partial_{0127} - \partial_{1467}$ & $\partial_{1236} + \partial_{2567}$\\\hline
weight & $\lambda^{-2}$ & $\lambda^{-92}$ & $\lambda^{20}$ & $\lambda^{112}$ & $\lambda^{110}$ & $\lambda^{-90}$
\end{tabular}$$
and the number of positive weights is 7.

At $\wil e_{0126}$ in $\wil U_{0126}$, the kernel of $\wil J_{0126}$ is generated by the following vectors
$$\begin{tabular}{c|c|c|c|c|c|c}
vector & $\partial_{0123}$ & $\partial_{0124}$ & $\partial_{0156}$ & $\partial_{0167}$ & $\partial_{0236}$ & $\partial_{0246}$\\\hline
weight & $\lambda^{108}$ & $\lambda^{200}$ & $\lambda^{-92}$ & $\lambda^{108}$ & $\lambda^{20}$ & $\lambda^{112}$
\end{tabular}$$
$$\begin{tabular}{c|c|c|c|c|c|c}
vector & $\partial_{0256}$ & $\partial_{1246}$ & $\partial_{1256}$ & $\partial_{1267}$ & $\partial_{0125} - \partial_{1236}$ & $\partial_{0146} + \partial_{0267}$\\\hline
weight & $\lambda^{-90}$ & $\lambda^{90}$ & $\lambda^{-112}$ & $\lambda^{88}$ & $\lambda^{-2}$ & $\lambda^{110}$
\end{tabular}$$
and the number of positive weights is 8.

At $\wil e_{0135}$ in $\wil U_{0135}$, the kernel of $\wil J_{0135}$ is generated by the following vectors
$$\begin{tabular}{c|c|c|c|c|c|c}
vector & $\partial_{0123}$ & $\partial_{0137}$ & $\partial_{0145}$ & $\partial_{0156}$ & $\partial_{0345}$ & $\partial_{0356}$\\\hline
weight & $\lambda^{92}$ & $\lambda^{200}$ & $\lambda^{92}$ & $\lambda^{-108}$ & $\lambda^{112}$ & $\lambda^{-88}$
\end{tabular}$$
$$\begin{tabular}{c|c|c|c|c|c|c}
vector & $\partial_{0357}$ & $\partial_{1235}$ & $\partial_{1356}$ & $\partial_{1357}$ & $\partial_{0136} + \partial_{0235}$ & $\partial_{0157} + \partial_{1345}$\\\hline
weight & $\lambda^{110}$ & $\lambda^{-20}$ & $\lambda^{-110}$ & $\lambda^{88}$ & $\lambda^{2}$ & $\lambda^{90}$
\end{tabular}$$
and the number of positive weights is 8.

At $\wil e_{1257}$ in $\wil U_{1257}$, the kernel of $\wil J_{1257}$ is generated by the following vectors
$$\begin{tabular}{c|c|c|c|c|c|c}
vector & $\partial_{1235}$ & $\partial_{1237}$ & $\partial_{1245}$ & $\partial_{1247}$ & $\partial_{1256}$ & $\partial_{1267}$\\\hline
weight & $\lambda^{-90}$ & $\lambda^{110}$ & $\lambda^{2}$ & $\lambda^{202}$ & $\lambda^{-198}$ & $\lambda^{2}$
\end{tabular}$$
$$\begin{tabular}{c|c|c|c|c|c|c}
vector & $\partial_{1357}$ & $\partial_{1457}$ & $\partial_{1567}$ & $\partial_{0125} + \partial_{2567}$ & $\partial_{0127} + \partial_{2457}$ & $\partial_{0157} - \partial_{2357}$\\\hline
weight & $\lambda^{18}$ & $\lambda^{110}$ & $\lambda^{-90}$ & $\lambda^{-88}$ & $\lambda^{112}$ & $\lambda^{20}$
\end{tabular}$$
and the number of positive weights is 8.

At $\wil e_{2456}$ in $\wil U_{2456}$, the kernel of $\wil J_{2456}$ is generated by the following vectors
$$\begin{tabular}{c|c|c|c|c|c|c}
vector & $\partial_{0246}$ & $\partial_{0256}$ & $\partial_{0456}$ & $\partial_{1245}$ & $\partial_{1246}$ & $\partial_{1256}$\\\hline
weight & $\lambda^{112}$ & $\lambda^{-90}$ & $\lambda^{20}$ & $\lambda^{88}$ & $\lambda^{90}$ & $\lambda^{-112}$
\end{tabular}$$
$$\begin{tabular}{c|c|c|c|c|c|c}
vector & $\partial_{2345}$ & $\partial_{2356}$ & $\partial_{2467}$ & $\partial_{4567}$ & $\partial_{0245} + \partial_{2346}$ & $\partial_{1456} - \partial_{2567}$\\\hline
weight & $\lambda^{108}$ & $\lambda^{-92}$ & $\lambda^{200}$ & $\lambda^{108}$ & $\lambda^{110}$ & $\lambda^{-2}$
\end{tabular}$$
and the number of positive weights is 8.

At $\wil e_{3567}$ in $\wil U_{3567}$, the kernel of $\wil J_{3567}$ is generated by the following vectors
$$\begin{tabular}{c|c|c|c|c|c|c}
vector & $\partial_{0356}$ & $\partial_{0357}$ & $\partial_{0367}$ & $\partial_{1356}$ & $\partial_{1357}$ & $\partial_{1567}$\\\hline
weight & $\lambda^{-88}$ & $\lambda^{110}$ & $\lambda^{112}$ & $\lambda^{-110}$ & $\lambda^{88}$ & $\lambda^{-20}$
\end{tabular}$$
$$\begin{tabular}{c|c|c|c|c|c|c}
vector & $\partial_{2356}$ & $\partial_{2367}$ & $\partial_{3457}$ & $\partial_{4567}$ & $\partial_{0567} + \partial_{3456}$ & $\partial_{1367} + \partial_{2357}$\\\hline
weight & $\lambda^{-108}$ & $\lambda^{92}$ & $\lambda^{200}$ & $\lambda^{92}$ & $\lambda^{2}$ & $\lambda^{90}$
\end{tabular}$$
and the number of positive weights is 8.

At $\wil e_{0256}$ in $\wil U_{0256}$, the kernel of $\wil J_{0256}$ is generated by the following vectors
$$\begin{tabular}{c|c|c|c|c|c|c}
vector & $\partial_{0126}$ & $\partial_{0156}$ & $\partial_{0236}$ & $\partial_{0246}$ & $\partial_{0356}$ & $\partial_{0456}$\\\hline
weight & $\lambda^{90}$ & $\lambda^{-2}$ & $\lambda^{110}$ & $\lambda^{202}$ & $\lambda^{18}$ & $\lambda^{110}$
\end{tabular}$$
$$\begin{tabular}{c|c|c|c|c|c|c}
vector & $\partial_{1256}$ & $\partial_{2356}$ & $\partial_{2456}$ & $\partial_{0125} + \partial_{2567}$ & $\partial_{0235} + \partial_{0567}$ & $\partial_{0245} - \partial_{0267}$\\\hline
weight & $\lambda^{-22}$ & $\lambda^{-2}$ & $\lambda^{90}$ & $\lambda^{88}$ & $\lambda^{108}$ & $\lambda^{200}$
\end{tabular}$$
and the number of positive weights is 9.

At $\wil e_{1235}$ in $\wil U_{1235}$, the kernel of $\wil J_{1235}$ is generated by the following vectors
$$\begin{tabular}{c|c|c|c|c|c|c}
vector & $\partial_{0123}$ & $\partial_{0135}$ & $\partial_{1237}$ & $\partial_{1245}$ & $\partial_{1256}$ & $\partial_{1257}$\\\hline
weight & $\lambda^{112}$ & $\lambda^{20}$ & $\lambda^{200}$ & $\lambda^{92}$ & $\lambda^{-108}$ & $\lambda^{90}$
\end{tabular}$$
$$\begin{tabular}{c|c|c|c|c|c|c}
vector & $\partial_{1356}$ & $\partial_{1357}$ & $\partial_{2345}$ & $\partial_{2356}$ & $\partial_{0125} - \partial_{1236}$ & $\partial_{1345} + \partial_{2357}$\\\hline
weight & $\lambda^{-90}$ & $\lambda^{108}$ & $\lambda^{112}$ & $\lambda^{-88}$ & $\lambda^{2}$ & $\lambda^{110}$
\end{tabular}$$
and the number of positive weights is 9.

At $\wil e_{1567}$ in $\wil U_{1567}$, the kernel of $\wil J_{1567}$ is generated by the following vectors
$$\begin{tabular}{c|c|c|c|c|c|c}
vector & $\partial_{0156}$ & $\partial_{0167}$ & $\partial_{1256}$ & $\partial_{1257}$ & $\partial_{1267}$ & $\partial_{1356}$\\\hline
weight & $\lambda^{-88}$ & $\lambda^{112}$ & $\lambda^{-108}$ & $\lambda^{90}$ & $\lambda^{92}$ & $\lambda^{-90}$
\end{tabular}$$
$$\begin{tabular}{c|c|c|c|c|c|c}
vector & $\partial_{1357}$ & $\partial_{1457}$ & $\partial_{3567}$ & $\partial_{4567}$ & $\partial_{0157} + \partial_{1367}$ & $\partial_{1456} - \partial_{2567}$\\\hline
weight & $\lambda^{108}$ & $\lambda^{200}$ & $\lambda^{20}$ & $\lambda^{112}$ & $\lambda^{110}$ & $\lambda^{2}$
\end{tabular}$$
and the number of positive weights is 9.

At $\wil e_{0156}$ in $\wil U_{0156}$, the kernel of $\wil J_{0156}$ is generated by the following vectors
$$\begin{tabular}{c|c|c|c|c|c|c}
vector & $\partial_{0126}$ & $\partial_{0135}$ & $\partial_{0145}$ & $\partial_{0167}$ & $\partial_{0256}$ & $\partial_{0356}$\\\hline
weight & $\lambda^{92}$ & $\lambda^{108}$ & $\lambda^{200}$ & $\lambda^{200}$ & $\lambda^{2}$ & $\lambda^{20}$
\end{tabular}$$
$$\begin{tabular}{c|c|c|c|c|c|c}
vector & $\partial_{0456}$ & $\partial_{1256}$ & $\partial_{1356}$ & $\partial_{1567}$ & $\partial_{0125} + \partial_{1456}$ & $\partial_{0136} - \partial_{0567}$\\\hline
weight & $\lambda^{112}$ & $\lambda^{-20}$ & $\lambda^{-2}$ & $\lambda^{88}$ & $\lambda^{90}$ & $\lambda^{110}$
\end{tabular}$$
and the number of positive weights is 10.

At $\wil e_{2356}$ in $\wil U_{2356}$, the kernel of $\wil J_{2356}$ is generated by the following vectors
$$\begin{tabular}{c|c|c|c|c|c|c}
vector & $\partial_{0236}$ & $\partial_{0256}$ & $\partial_{0356}$ & $\partial_{1235}$ & $\partial_{1256}$ & $\partial_{1356}$\\\hline
weight & $\lambda^{112}$ & $\lambda^{2}$ & $\lambda^{20}$ & $\lambda^{88}$ & $\lambda^{-20}$ & $\lambda^{-2}$
\end{tabular}$$
$$\begin{tabular}{c|c|c|c|c|c|c}
vector & $\partial_{2345}$ & $\partial_{2367}$ & $\partial_{2456}$ & $\partial_{3567}$ & $\partial_{0235} - \partial_{3456}$ & $\partial_{1236} + \partial_{2567}$\\\hline
weight & $\lambda^{200}$ & $\lambda^{200}$ & $\lambda^{92}$ & $\lambda^{108}$ & $\lambda^{110}$ & $\lambda^{90}$
\end{tabular}$$
and the number of positive weights is 10.

At $\wil e_{1356}$ in $\wil U_{1356}$, the kernel of $\wil J_{1356}$ is generated by the following vectors
$$\begin{tabular}{c|c|c|c|c|c|c}
vector & $\partial_{0135}$ & $\partial_{0156}$ & $\partial_{0356}$ & $\partial_{1235}$ & $\partial_{1256}$ & $\partial_{1357}$\\\hline
weight & $\lambda^{110}$ & $\lambda^{2}$ & $\lambda^{22}$ & $\lambda^{90}$ & $\lambda^{-18}$ & $\lambda^{198}$
\end{tabular}$$
$$\begin{tabular}{c|c|c|c|c|c|c}
vector & $\partial_{1567}$ & $\partial_{2356}$ & $\partial_{3567}$ & $\partial_{0136} + \partial_{3456}$ & $\partial_{1236} + \partial_{1456}$ & $\partial_{1345} - \partial_{1367}$\\\hline
weight & $\lambda^{90}$ & $\lambda^{2}$ & $\lambda^{110}$ & $\lambda^{112}$ & $\lambda^{92}$ & $\lambda^{200}$
\end{tabular}$$
and the number of positive weights is 11.

\nocite{AK16,AS08,AS10,AS08Cal,Kar06,CG83,Zho05,FH13,Kna13,Hum12,Bor12,Joy00}
\bibliographystyle{alpha}
\bibliography{yildirim}

\begin{thebibliography}{BBCM02}

\bibitem[AC15]{AC15}
Selman Akbulut and Mahir~Bilen Can.
\newblock Complex $\text{G}_2$ and associative grassmannian.
\newblock {\em arXiv preprint arXiv:1512.03191}, 2015.

\bibitem[AK16]{AK16}
Selman Akbulut and Mustafa Kalafat.
\newblock Algebraic topology of manifolds.
\newblock {\em Expositiones Mathematicae}, 34(1):106--129, 2016.

\bibitem[AS08a]{AS08Cal}
Selman Akbulut and Sema Salur.
\newblock Calibrated manifolds and gauge theory.
\newblock {\em Journal f{\"u}r die reine und angewandte Mathematik (Crelles
  Journal)}, 2008(625):187--214, 2008.

\bibitem[AS08b]{AS08}
Selman Akbulut and Sema Salur.
\newblock Deformations in $\text{G}_2$ manifolds.
\newblock {\em Advances in Mathematics}, 217(5):2130--2140, 2008.

\bibitem[AS10]{AS10}
Selman Akbulut and Sema Salur.
\newblock Mirror duality via {G$_2$} and {Spin(7)} manifolds.
\newblock {\em Arithmetic and Geometry Around Quantization}, pages 1--21, 2010.

\bibitem[Bae02]{Bae02}
John Baez.
\newblock The octonions.
\newblock {\em Bulletin of the American Mathematical Society}, 39(2):145--205,
  2002.

\bibitem[BB73]{BB73}
A.~Bialynicki-Birula.
\newblock Some theorems on actions of algebraic groups.
\newblock {\em Annals of Mathematics}, 98(3):480--497, 1973.

\bibitem[BBCM02]{BCM02}
Andrzej Bia{\l}ynicki-Birula, James~B Carrell, and William~M McGovern.
\newblock {\em Algebraic quotients torus actions and cohomology the adjoint
  representation and the adjoint action}.
\newblock Springer, 2002.

\bibitem[BG67]{BG67}
Robert~B Brown and Alfred Gray.
\newblock Vector cross products.
\newblock {\em Commentarii Mathematici Helvetici}, 42(1):222--236, 1967.

\bibitem[Bor12]{Bor12}
Armand Borel.
\newblock {\em Linear algebraic groups}, volume 126.
\newblock Springer Science \& Business Media, 2012.

\bibitem[Bry87]{Bry87}
Robert~L Bryant.
\newblock Metrics with exceptional holonomy.
\newblock {\em Annals of mathematics}, pages 525--576, 1987.

\bibitem[CG83]{CG83}
JB~Carrell and RM~Goresky.
\newblock A decomposition theorem for the integral homology of a variety.
\newblock {\em Inventiones mathematicae}, 73(3):367--381, 1983.

\bibitem[FH13]{FH13}
William Fulton and Joe Harris.
\newblock {\em Representation theory: a first course}, volume 129.
\newblock Springer Science \& Business Media, 2013.

\bibitem[HL82]{HL82}
Reese Harvey and H~Blaine Lawson.
\newblock Calibrated geometries.
\newblock {\em Acta Mathematica}, 148(1):47--157, 1982.

\bibitem[Hor69]{Hor69}
Geoffrey Horrocks.
\newblock Fixed point schemes of additive group actions.
\newblock {\em Topology}, 8(3):233--242, 1969.

\bibitem[Hum12]{Hum12}
James~E Humphreys.
\newblock {\em Linear algebraic groups}, volume~21.
\newblock Springer Science \& Business Media, 2012.

\bibitem[Joy00]{Joy00}
Dominic~D Joyce.
\newblock {\em Compact manifolds with special holonomy}.
\newblock Oxford University Press on Demand, 2000.

\bibitem[Kar06]{Kar06}
Spiro Karigiannis.
\newblock Some notes on $\text{G}_2$ and $\text{Spin}$(7) geometry.
\newblock {\em arXiv preprint math/0608618}, 2006.

\bibitem[KL72]{KL72}
S.~L. Kleiman and Dan Laksov.
\newblock Schubert calculus.
\newblock {\em The American Mathematical Monthly}, 79(10):1061--1082, 1972.

\bibitem[Kna13]{Kna13}
Anthony~W Knapp.
\newblock {\em Lie groups beyond an introduction}, volume 140.
\newblock Springer Science \& Business Media, 2013.

\bibitem[SV13]{SV13}
Tonny~A Springer and Ferdinand~D Veldkamp.
\newblock {\em Octonions, Jordan algebras and exceptional groups}.
\newblock Springer, 2013.

\bibitem[SW10]{SW10}
Dietmar~A Salamon and Thomas Walpuski.
\newblock Notes on the octonions.
\newblock {\em arXiv preprint arXiv:1005.2820}, 2010.

\bibitem[Zho05]{Zho05}
Jianwei Zhou.
\newblock Morse functions on grassmann manifolds.
\newblock {\em Proceedings of the Royal Society of Edinburgh Section A:
  Mathematics}, 135(1):209--221, 2005.

\end{thebibliography}

\end{document}